\documentclass[11pt]{article}
\usepackage{tikz}
\usepackage{tikz-qtree}
\usepackage[page, header]{appendix}
\usepackage{titletoc}
\usepackage{amsmath, amsfonts, amssymb, amsthm, amsbsy, amscd, bm, bbm,mathrsfs}
\usepackage{color}
\usepackage[colorlinks,linkcolor=blue,citecolor=blue]{hyperref}
\usepackage[top=1.5cm, bottom=1.5cm, outer=3cm, inner=3cm, heightrounded, marginparwidth=3cm, marginparsep=3cm]{geometry}

\usepackage{comment}
\usepackage{mathtools}
\mathtoolsset{showonlyrefs}

\renewcommand{\leq}{\leqslant}
\renewcommand{\geq}{\geqslant}

\newtheorem{theorem}{Theorem}

\newtheorem{lemma}{Lemma}
\newtheorem{proposition}{Proposition}
\theoremstyle{definition}

\newtheorem{remark}{Remark}

\providecommand{\keywords}[1]
{
	\small	
	\textbf{\textit{Keywords---}} #1
}

\providecommand{\MSC}[2]
{
	\small	
	\textbf{\textit{2020 MSC---}} #1
}

\title{Tail probability of maximal displacement in critical and subcritical branching stable processes}
\author{ \bf  Haojie Hou\footnote{The research of this author is supported by the China Postdoctoral Science Foundation (No. 2024M764112).}
 \hspace{1mm}\hspace{1mm}
	Yiyang Jiang \hspace{1mm}\hspace{1mm}
Yan-Xia Ren\footnote{The research of this author is supported by NSFC (Grant No 12231002) and The Fundamental
Research Funds for the Central Universities, Peking University LMEQF.
 } \hspace{1mm}\hspace{1mm} and \hspace{1mm}\hspace{1mm}
Renming Song\thanks{Research supported in part by a grant from the Simons
Foundation
(\#960480, Renming Song).}
\hspace{1mm} }
	
\begin{document}
		\maketitle
		\begin{abstract}
		In this paper, we study critical and subcritical branching $\alpha$-stable processes, $\alpha \in (0, 2)$. We obtain the exact asymptotic behaviors of the tails of the maximal positions of
		all subcritical branching $\alpha$-stable processes with positive jumps.
		In the case of subcritical branching spectrally negative $\alpha$-stable processes, we obtain the exact asymptotic behaviors of the tails of the maximal positions under the assumption that the offspring distributions satisfy the $L\log L$ condition.
		For critical branching $\alpha$-stable processes, we obtain the exact asymptotic behaviors of the tails under the assumption that the offspring distributions belong to the domain of attraction of a $\gamma$-distribution,  $\gamma\in (1, 2]$.
		\end{abstract}\hspace{10pt}
		

       \MSC{60J80; 60G52; 60G51; 60G70.}

		\section{Introduction and main results}
		\subsection{Introduction}

       A branching L\'evy process is a continuous-time Markov process defined as follows. At time 0, there is a particle at $x\in \mathbb{R}$ and it moves according to a L\'evy process $(\xi_t, \mathbf{P}_x)$. After an exponential time with parameter $1$, independent of the motion, it dies and produces $k$ offspring with probability $p_k$, $k\geq 0$. The offspring move independently according to $\xi$ from the place where they are born and obey the same branching mechanism as their parent. This procedure goes on. For $t\geq 0$, let $N_t$ be the collection of particles alive at time $t$. For $u\in N_t$, we use $X_u(t)$ to denote the position of the particle $u$ at time $t$. The point process $(Z_t)_{t\geq 0}$ defined by
      $$
     Z_t:=\sum_{u\in N_t}\delta_{X_u(t)}, \qquad t \geq 0,
   $$
	is called a branching L\'evy process. We will denote the law of $(Z_t)_{t\geq 0}$ by $\mathbb{P}_x$ and write $\mathbb{P}$ for $\mathbb{P}_0$ for simplicity.
	When $\xi$ is a Brownian motion, $(Z_t)_{t\geq 0}$ is called a branching Brownian motion. When $\xi$ is 	an $\alpha$-stable process, $\alpha\in (0, 2)$,  $(Z_t)_{t\geq 0}$ is called a branching $\alpha$-stable process.
		
Let $m:=\sum^\infty_{k=0}kp_k$ be the mean of the offspring distribution. When $m>1(=1, <1)$, we say that
the branching L\'evy process $(Z_t)_{t\geq 0}$ is supercritical (critical, subcritical). It is well known that in the critical and subcritical cases, $(Z_t)_{t\geq 0}$ will die out in finite time with probability 1. Thus in this case we can define the maximal position of $(Z_t)_{t\geq 0}$ by
$$
M:=\sup_{t>0}\sup_{u\in N_t} X_u(t).
$$

When $m=1$ and $\xi$ is a standard Brownian motion, Sawyer and Fleischman \cite{Sawyer} proved that,  under the assumption $\sum_{k=0}^\infty k^3 p_k<\infty$,
		\begin{align}\label{Tail-probability1}
			\lim_{x\to +\infty} x^2 \mathbb{P}\left(M\geq x\right)  =\frac{6}{\sigma^2 },
		\end{align}
where $\sigma^2$ is the variance of the offspring distribution. Profeta \cite{profeta:hal-03736526} extended the result above to some
critical branching spectrally negative L\'evy processes under the same third moment condition on the offspring distribution. When $m=1$, $\xi$ is an $\alpha$-stable process with positive jumps and $\sum_{k=0}^\infty k^3 p_k<\infty$, Lalley and Shao \cite{lalley}, and Profeta \cite{profeta}
proved that
\begin{align} \label{result1}
		\lim_{x\to +\infty} x^{\alpha/2}\mathbb{P}\left(M\geq x\right)  =c_1(\alpha),
\end{align}
where $c_1(\alpha)$ is an explicit positive constant. For critical branching spectrally negative $\alpha$-stable processes
with
 $\alpha\in (1, 2)$,  Profeta \cite[Corollary 1.3]{profeta:hal-03736526} proved that,
when $\sum_{k=0}^\infty k^3 p_k<\infty$,
\begin{align}\label{Profeta-bound}
        	0< \liminf_{x\to\infty} x^{\alpha} \mathbb{P}\left(M\geq x\right)\leq \limsup_{x\to\infty} x^{\alpha} \mathbb{P}\left(M \geq x\right)<\infty.
        \end{align}

 When $m<1$ and $\xi$ is a standard Brownian motion, Sawyer and Fleischman \cite{Sawyer} proved that,  under the assumption $\sum_{k=0}^\infty k^3 p_k<\infty$, there exists a function $c(\cdot)$ bounded between two positive constants such that
\begin{align}\label{Tail-probability2}
			\lim_{x\to +\infty} \frac{\mathbb{P}\left(M\geq x\right)}{(1-m)c(x)e^{-\sqrt{2(1-m)}x}}=1.
		\end{align}
For subcritical branching $\alpha$-stable processes with positive jumps,
Profeta \cite{profeta} proved that,
 under the assumption $\sum_{k=0}^\infty k^3 p_k<\infty$,
		\begin{align}\label{result2}
		\lim_{x\to +\infty} x^{\alpha}\mathbb{P}\left(M\geq x\right)=
     \frac{c_2(\alpha)}{1-m}
		\end{align}
for some explicit constant $c_2(\alpha)$. 		
For subcritical branching spectrally negative processes,
Profeta \cite[Theorem 1]{profeta:hal-03736526}
proved that, under the assumption $\sum_{k=0}^\infty k^3 p_k<\infty$,
 \begin{align}\label{Profeta-limit}
    \lim_{x\to\infty} e^{c_4x}  \mathbb{P}\left(M\geq x\right) = c_3
      \end{align}
  for some  explicit positive constants $c_3, c_4$.

 For results on the asymptotic behaviors of the tails of maximal positions of (sub)critical branching random walks, see \cite{lalleywalk, neuman2017maximal}.

 The purpose of this paper is  to
 establish the exact asymptotic behaviors of the tails of the maximal positions of (sub)critical branching $\alpha$-stable processes under minimal conditions on the offspring distributions. More precisely, we will obtain (i) the exact asymptotic behaviors of the tails of the maximal positions of subcritical branching $\alpha$-stable processes with positive jumps without any extra assumption on the offspring distributions; (ii) the exact asymptotic behaviors of the tail of the maximal positions of subcritical branching spectrally negative $\alpha$-stable processes under the assumption that the offspring distributions satisfy the $L\log L$ condition; (iii) the exact asymptotic behaviors of the tails of critical branching $\alpha$-stable processes under the assumption that the offspring distributions belong to the domain of attraction of a $\gamma$-distribution,  $\gamma\in (1, 2]$.

		\subsection{Main results}
		
Before we state our main results, we recall some useful facts about stable processes.
In this paper, we always assume that the spatial motion $\xi$ is a (strictly) $\alpha$-stable process, $\alpha\in (0, 2)$. For basic information on $\alpha$-stable process, see, for instance, \cite[Chap. VIII]{bertoin}. We assume that the L\'{e}vy measure of $\xi$ is given by
		\begin{equation}\label{def-v}
			v_\alpha(\mathrm{d}x):=c_+x^{-(1+\alpha)}{\bf 1}_{(0,\infty)}(x){\rm d} x+c_-|x|^{-(1+\alpha)}{\bf 1}_{(-\infty,0)}(x){\rm d} x,
		\end{equation}
where $c_+$ and $c_-$ are non-negative numbers with at least one of them being positive.
Let
$\Psi(\theta):=-\ln \mathbf{E}(e^{\mathrm{i}\theta \xi_{1} })$
 be the characteristic exponent of $\xi$. It is known (see, for instance  \cite[Chapter 1.2.6]{kyprianou2014fluctuations}) that, 		
$$
		\Psi(\theta)=
		\begin{cases}
			c_*|\theta|^\alpha\left(1-\mathrm{i} \beta \tan \frac{\pi \alpha}{2} \operatorname{sgn} \theta\right),
			& \text { for } \alpha \in(0,1) \cup(1,2) \text { and } \beta \in[-1,1],\\
			c_*|\theta|+\mathrm{i} \theta \eta, & \text { for } \alpha=1
		\end{cases}
		$$
		where $ \eta \in \mathbb{R}$, $c_*:=-\left(c_++c_-\right) \Gamma(-\alpha) \cos (\pi \alpha / 2)$, and $\beta=\left(c_+-c_-\right) /\left(c_++c_-\right)$ if $\alpha \in(0,1) \cup(1,2)$. Here $\operatorname{sgn} \theta=1_{(\theta>0)}-1_{(\theta<0)}$.
		Note that, for (strictly) 1-stable process, $c_+=c_-$
	(see, for example, \cite[Theorem 14.7 (v)]{Sato}).
It follows from \cite[Example 1.1, Lemma 2.1 and the last  sentence in Subsection 1.1]{Ren_Song_Zhang_2023}	 (see also \cite[Proposition VIII.4]{bertoin}) that,  when $c_+>0$,
		\begin{align}\label{Tail-of-xi}
			\lim_{x\to\infty} x^{\alpha} \mathbf{P}_0(\xi_1\geq x) = \nu_\alpha ((1,\infty))= \frac{c_+}{\alpha}.
		\end{align}
 When $c_+=0$ and $\alpha\in (0,1)$, $-\xi$ is a subordinator and $M=0$ almost surely.
A spectrally negative 1-stable process reduces to the drift $-\eta t$. When $\eta\geq 0$, $M=0$.
Thus in the case of branching spectrally negative $\alpha$-stable processes,
 exclude the above cases and assume
either $\alpha\in (1,2)$, or $\alpha=1$ and $\eta<0$.
In this case, it is easy to see that
for any $\lambda>0$,
		\begin{align}\label{Def-of-Upsilon}
			 \mathbf{E}_0(e^{\lambda \xi_1})=
			 \exp\left\{C_1(\alpha) \lambda^\alpha \right\},
			 \quad \mbox{where}\quad 	C_1(\alpha)=
			-\Psi(-\mathrm{i}).
		\end{align}
		Define $\tau_{y} = \inf\{t>0, \xi_t\geq y\}$. For any $x<y$, combining \eqref{Def-of-Upsilon} with  the fact that
		$\mathbf{P}_x\left(\xi_{\tau_y}=y\right)=1$,
		 we have that for any $\lambda>0$,
		\begin{align}\label{Laplace-of-stopping-time}
			\mathbf{E}_x\left(e^{-\lambda \tau_y} \right) = \exp\left\{-
			C_1(\alpha)^{-1/\alpha}\lambda^{1/\alpha}
					(y-x)\right\}.
		\end{align}		

In the critical case, we will need the following assumption on the offspring distribution:
\begin{itemize}
			\item [{\bf(H)}]

			The offspring distribution $\{p_k: k\geq 0\}$ belongs to the domain of attraction of a $\gamma$-stable, $\gamma\in (1,2]$, distribution.
			More precisely, either there exist $\gamma\in (1,2)$ and $\kappa_{\gamma}\in(0,\infty)$ such that
			\[
			\lim_{n\to\infty} n^\gamma \sum_{k=n}^\infty p_k = \kappa_{\gamma},
			\]
			or that (corresponding $\gamma=2$)
			\[
			\sum_{k=0}^\infty k^2 p_k<\infty.
			\]
		\end{itemize}
		
		Recall that $\sigma^2$ is the variance of the offspring distribution. We define
		\begin{align}\label{Def-of-Lambda-gamma}
			C_2(\gamma): =\begin{cases}
				\frac{\Gamma(2-\gamma)}{\gamma-1}\kappa_{\gamma} & \text { for } \gamma \in(1,2);\\
				\frac{1}{2}\sigma^2& \text { for } \gamma=2.
			\end{cases}
		\end{align}

Our first main result is on (sub)critical branching stable processes with positive jumps.

\begin{theorem}\label{t1}
	Suppose that $c_+>0$.
	
	i) If $m<1$, then
	\begin{align}
		\lim_{x\to +\infty} x^{\alpha}\mathbb{P}\left(M\geq x\right)  =\frac{c_+}{(1-m)\alpha}.
	\end{align}
	
	ii) If $m=1$ and  {\bf(H)} holds, then
	$$
	\lim_{x\rightarrow+\infty}x^{\alpha/\gamma} \mathbb{P}\left(M \geq x\right)=\left(\frac{c_+}{\alpha C_2(\gamma)}\right)^{\frac{1}{\gamma}},
	$$
	where $C_2(\gamma)$ is defined in \eqref{Def-of-Lambda-gamma}.
\end{theorem}

		\begin{remark}
Theorem \ref{t1} (ii) is also contained in  Theorem 3 of the recent preprint \cite{Profeta2025}.
Since there are  some differences between the proofs of Theorem \ref{t1} (ii) and \cite[Theorem 3]{Profeta2025}, we include it here for
completeness.
The proof of Theorem \ref{t1} is an adaptation of that of the corresponding result in \cite{profeta}
below. {\bf(H)} only changes the behaviors of $\Phi_0$ and $\Phi_R$, defined in  \eqref{main}.
 In \cite[(1.5)]{profeta},  the third moment condition on the offspring distribution is used to estimate $\Phi_R$ directly while in the $\gamma$-stable branching case, we have to do a more careful analysis.
		\end{remark}

Our second main result is
on the case for critical branching spectrally negative $\alpha$-stable process with $\alpha\in (1,2)$, and it generalizes and refines \eqref{Profeta-bound}.

	\begin{theorem}\label{t2}
Suppose $c_+=0$ and $ \alpha\in (1,2)$
or $\alpha=1$ and $\eta <0$.
 If $m=1$ and  {\bf(H)} holds,
then there exists a constant $C_3(\alpha,\beta,\gamma)\in (0,\infty)$ such that
		\begin{align}
			\lim_{x\to\infty} x^{\frac{\alpha}{\gamma-1}}  \mathbb{P}\left(M \geq x\right) = C_3(\alpha, \beta, \gamma).
		\end{align}
	\end{theorem}
	
	\begin{remark}
		The constant $C_3(\alpha,\beta,\gamma)$ has a probabilistic representation via a super $\alpha$-stable process, see the proof of Proposition \ref{prop1}.
    In the proof of Proposition \ref{prop1}, we use the fact that a superprocess is an appropriate scaling limit of branching Markov processes.
	\end{remark}

Our last main result is on subcritical branching spectrally negative $\alpha$-stable process with $\alpha\in (1,2)$.

\begin{theorem}\label{t3}
Suppose that  $c_+=0$, and that either $\alpha\in (1,2) $
or $\alpha=1$ and $\eta<0$.
If $m<1$ and $\sum_{k=0}^\infty k(\log k)p_k<\infty$,
	then there exists a positive constant $C_4(\alpha)$ such that
	 \begin{align}
		\lim_{x\to\infty} e^{((1-m)/C_1(\alpha))^{1/\alpha}x}  \mathbb{P}\left(M\geq x\right) = C_4(\alpha).
	\end{align}
\end{theorem}
	
\begin{remark}
	Our proof of  Theorem \ref{t3} can be easily adapted to a widely class of subcritical branching spectrally negative L\'{e}vy processes.
\end{remark}

 \section{An integral equation for $\mathbb{P}\left(M\geq x\right)$}

Let
\begin{align}
	u(x):=\mathbb{P}\left(M \geq x\right)\quad \mbox{and}\quad S_t:=\sup _{s \in[0, t]} \xi_s
\end{align}
be the tail probability of $M$ and the supremum of $\xi$ up to time $t$. Let $\mathbf{e}$ be an exponential random variable with  parameter $1$  independent of $\xi$.
Define
\begin{align}\label{Def-of-G}
	G(x):=
	\sum_{k=0}^\infty p_k(1-x)^k - 1+mx,\quad x\in [0,1].
\end{align}
It is easy to see that $G$ is a
 non-negative function in $[0,1]$.

		\begin{proposition}\label{Proposition-inte}
			The function $u$ is a solution of the integral
			equation:
			\begin{equation}\label{integral equation}
				u(x)=\mathbf{P}_0\left(S_{\mathbf{e}} \geq x\right)+\mathbf{E}_0\left[1_{\left\{S_{\mathbf{e}}<x\right\}} u\left(x-\xi_{\mathbf{e}}\right)\right]-\Phi_0(x)-\Phi_R(x),
			\end{equation}
			where
			\begin{equation}\label{main}
				\Phi_0(x)=(1-m)\mathbf{E}_0\left[1_{\left\{S_{\mathbf{e}}<x\right\}} u\left(x-\xi_{\mathbf{e}}\right)\right]\quad \mbox{and}\quad 	\Phi_R(x)=\mathbf{E}_0\left[1_{\left\{S_{\mathbf{e}}<x\right\}} G(u\left(x-\xi_{\mathbf{e}}\right))\right].
			\end{equation}

(i) If $m=1$ and {\bf(H)} holds,
	then for any $\varepsilon>0$, there exists $\delta>0$ such that
						\begin{equation}\label{remainder1.1}
				\Phi_R(x) \geq  (1-\varepsilon)C_2(\gamma)\mathbf{E}_0\left[1_{\left\{S_{\mathbf{e}}<x\right\}} u^{\gamma}\left(x-\xi_{\mathbf{e}}\right)\right]-\frac{ (1-\varepsilon)C_2(\gamma)}{\delta}\mathbf{E}_0\left[1_{\left\{S_{\mathbf{e}}<x\right\}} u^{\gamma+1}\left(x-\xi_{\mathbf{e}}\right)\right]
			\end{equation}
					and
			\begin{equation}\label{remainder1.2}
				\Phi_R(x)\leq (1+\varepsilon)C_2(\gamma)\mathbf{E}_0\left[1_{\left\{S_{\mathbf{e}}<x\right\}} u^{\gamma}\left(x-\xi_{\mathbf{e}}\right)\right]+\frac{1}{\delta^{\gamma+1}}\mathbf{E}_0\left[1_{\left\{S_{\mathbf{e}}<x\right\}} u^{\gamma+1}\left(x-\xi_{\mathbf{e}}\right)\right],
			\end{equation}
			where $C_2(\gamma)$ is defined in \eqref{Def-of-Lambda-gamma}.
			
      (ii)   If $m<1$ and $c_+>0$, then
			\begin{equation}\label{remainder2}
				\lim_{x\to\infty}\frac{\Phi_R(x)}{\Phi_0(x)}=0.
			\end{equation}
		\end{proposition}
		
		\begin{proof}
			Applying the Markov property at the first branching time, we get
			$$
		     \mathbb{P}\left(M<x\right)=p_0 \mathbf{P}\left(S_{\mathbf{e}}<x\right)+\sum_{n=1}^{+\infty} p_n \mathbb{P}\left(S_{\mathbf{e}}<x, \xi_{\mathbf{e}}+M^{(1)}<x, \ldots, \xi_{\mathbf{e}}+M^{(n)}<x\right),
			$$
			where $\left(M^{(n)}\right)_{n \in \mathbb{N}}$ are independent copies of $M$, which are also independent of  $\left(\xi_{\mathbf{e}}, S_{\mathbf{e}}\right)$. Hence,
			\begin{align*}
				1-u(x)&=p_0 \mathbf{P}_0\left(S_{\mathbf{e}}<x\right)+\sum_{n=1}^{+\infty} p_n \mathbf{E}_0\left[1_{\left\{S_{\mathbf{e}}<x\right\}}\left(1-u\left(x-\xi_{\mathbf{e}}\right)\right)^n\right] \\
				&=p_0 \mathbf{P}_0\left(S_{\mathbf{e}}<x\right)-m\mathbf{E}_0\left[1_{\left\{S_{\mathbf{e}}<x\right\}} u\left(x-\xi_{\mathbf{e}}\right)\right]\\
				&\quad+ \mathbf{E}_0\left[1_{\left\{S_{\mathbf{e}}<x\right\}}\left(\sum_{n=1}^{+\infty} p_n\left(1-u\left(x-\xi_{\mathbf{e}}\right)\right)^n+mu\left(x-\xi_{\mathbf{e}}\right)\right)\right].
			\end{align*}
Therefore, 		
			\begin{align*}
				u(x)=&\mathbf{P}_0\left(S_{\mathbf{e}} \geq x\right)+\mathbf{E}_0\left[1_{\left\{S_{\mathbf{e}}<x\right\}} u\left(x-\xi_{\mathbf{e}}\right)\right]-(1-m)\mathbf{E}_0\left[1_{\left\{S_{\mathbf{e}}<x\right\}} u\left(x-\xi_{\mathbf{e}}\right)\right]\\
				&-\mathbf{E}_0\left[1_{\left\{S_{\mathbf{e}}<x\right\}}\left(\sum_{n=1}^{+\infty} p_n\left(1-u\left(x-\xi_{\mathbf{e}}\right)\right)^n+mu\left(x-\xi_{\mathbf{e}}\right)-(1-p_0)\right)\right]\\
				=&\mathbf{P}_0\left(S_{\mathbf{e}} \geq x\right)+\mathbf{E}_0\left[1_{\left\{S_{\mathbf{e}}<x\right\}} u\left(x-\xi_{\mathbf{e}}\right)\right]-\Phi_0(x)-\Phi_R(x),
			\end{align*}
where $\Phi_0$ and $\Phi_R$ are given in \eqref{main}.

  We  prove (i) first.
 		When $m=1$ and {\bf(H)} holds,
				from \cite[Lemma 3.1]{HJRS2025} (for $\gamma\in (1, 2)$) and L'Hopital's rule (for $\gamma=2$), we get that 		
		\begin{align}\label{step_2}
			\lim_{u\downarrow 0}
			\frac{G(u)}{u^{\gamma}} = C_2(\gamma).
		\end{align}
		Therefore, for any $\varepsilon>0$, there exists $\delta>0$ such that for all $u\leq\delta$,
			\begin{equation}\label{generating}
				(1-\varepsilon)C_2(\gamma)\leq \frac{G(u)}{u^{\gamma}}\leq (1+\varepsilon)C_2(\gamma).
			\end{equation}
			Plugging \eqref{generating} into the definition of $\Phi_R$ in \eqref{main}, we get
			\begin{align*}
				\Phi_R(x)&\geq \mathbf{E}_0\left[1_{\left\{S_{\mathbf{e}}<x\right\}} G(u\left(x-\xi_{\mathbf{e}}\right))1_{\left\{u\left(x-\xi_{\mathbf{e}}\right)<\delta\right\}}\right]\\
				&\geq (1-\varepsilon)C_2(\gamma)\mathbf{E}_0\left[1_{\left\{S_{\mathbf{e}}<x\right\}} u^{\gamma}\left(x-\xi_{\mathbf{e}}\right)1_{\left\{u\left(x-\xi_{\mathbf{e}}\right)<\delta\right\}}\right]\\
				&\geq (1-\varepsilon)C_2(\gamma)\mathbf{E}_0\left[1_{\left\{S_{\mathbf{e}}<x\right\}} u^{\gamma}\left(x-\xi_{\mathbf{e}}\right)\right]-\frac{(1-\varepsilon)C_2(\gamma)}{\delta}\mathbf{E}_0\left[1_{\left\{S_{\mathbf{e}}<x\right\}} u^{\gamma+1}\left(x-\xi_{\mathbf{e}}\right)\right],
			\end{align*}
			where in the last inequality we used $1_{\left\{u<\delta\right\}}=1-1_{\left\{u\geq\delta\right\}}$ and $1_{\left\{u\geq\delta\right\}}\leq u/\delta$.
			Thus  \eqref{remainder1.1} is valid.
            On the other hand, since $G(x)\leq 1$ for $x\in [0,1]$, we have
			\begin{align}\label{upper-of-phi-r}
				\Phi_R(x)&= \mathbf{E}_0\left[1_{\left\{S_{\mathbf{e}}<x\right\}} G(u\left(x-\xi_{\mathbf{e}}\right))1_{\left\{u\left(x-\xi_{\mathbf{e}}\right)<\delta\right\}}\right]+\mathbf{E}_0\left[1_{\left\{S_{\mathbf{e}}<x\right\}} G(u\left(x-\xi_{\mathbf{e}}\right))1_{\left\{u\left(x-\xi_{\mathbf{e}}\right)\geq\delta\right\}}\right]\nonumber \\
				&\leq (1+\varepsilon)C_2(\gamma)\mathbf{E}_0\left[1_{\left\{S_{\mathbf{e}}<x\right\}} u^{\gamma}\left(x-\xi_{\mathbf{e}}\right)\right]+\frac{1}{\delta^{\gamma+1}}\mathbf{E}_0\left[1_{\left\{S_{\mathbf{e}}<x\right\}} u^{\gamma+1}\left(x-\xi_{\mathbf{e}}\right)\right],
			\end{align}
			which implies \eqref{remainder1.2}.

	We now consider the case $m<1$ and prove (ii) under the assumption $c_+>0$.
	Note that for any $\varepsilon>0$, there exists $\delta>0$ such that
			\[
		  G(u)\leq \varepsilon u, \qquad \mbox{ for all } u\leq \delta.
			\]
			Similar to \eqref{upper-of-phi-r}, using the fact that $G(x)\leq 1$, we get that
			\begin{align}
				 	\Phi_R(x)&= \mathbf{E}_0\left[1_{\left\{S_{\mathbf{e}}<x\right\}} G(u\left(x-\xi_{\mathbf{e}}\right))1_{\left\{u\left(x-\xi_{\mathbf{e}}\right)<\delta\right\}}\right]+\mathbf{E}_0\left[1_{\left\{S_{\mathbf{e}}<x\right\}} G(u\left(x-\xi_{\mathbf{e}}\right))1_{\left\{u\left(x-\xi_{\mathbf{e}}\right)\geq\delta\right\}}\right]\\
				 	&\leq \varepsilon \mathbf{E}_0\left[1_{\left\{S_{\mathbf{e}}<x\right\}} u\left(x-\xi_{\mathbf{e}}\right)\right]+\frac{1}{\delta^{2}}\mathbf{E}_0\left[1_{\left\{S_{\mathbf{e}}<x\right\}} u^{2}\left(x-\xi_{\mathbf{e}}\right)\right].
			\end{align}
			Therefore, to prove \eqref{remainder2}, it suffices to show that
		\begin{equation}\label{suffice}
			\lim_{x\to\infty}\frac{\mathbf{E}_0\left[1_{\left\{S_{\mathbf{e}}<x\right\}} u^{2}\left(x-\xi_{\mathbf{e}}\right)\right]}{\mathbf{E}_0\left[1_{\left\{S_{\mathbf{e}}<x\right\}} u\left(x-\xi_{\mathbf{e}}\right)\right]}
			=0.
		\end{equation}
		Considering the case where the initial particle does not split before time 1 and using \eqref{Tail-of-xi}, we get
		$$
		u(x)\geq e^{-1}\mathbf{P}_0(S_1>x)\geq e^{-1}\mathbf{P}_0(\xi_1>x)\geq \frac{c_+}{2e\alpha}x^{-\alpha}.
		$$
when $x$ is large enough. Therefore, when $x$ is large enough,
			\begin{align}\label{de}
				\mathbf{E}_0\left[1_{\left\{S_{\mathbf{e}}<x\right\}} u\left(x-\xi_{\mathbf{e}}\right)\right]
				&\geq \mathbf{E}_0\left[1_{\left\{S_{\mathbf{e}}<x\right\}} u\left(x-\xi_{\mathbf{e}}\right)1_{\left\{|\xi_{\mathbf{e}}|<\frac{1}{2}x\right\}}\right] \nonumber\geq  \frac{c_+}{2e\alpha} \mathbf{E}_0\left[1_{\left\{S_{\mathbf{e}}<x\right\}}\left(x-\xi_{\mathbf{e}}\right)^{-\alpha
				}1_{\left\{|\xi_{\mathbf{e}}|<\frac{1}{2}x\right\}}\right]\\
				&\geq \frac{c_+}{2e\alpha} \left(\frac{3}{2}x\right)^{-\alpha}\mathbf{P}_0\left(S_{\mathbf{e}}<x, |\xi_{\mathbf{e}}|<\frac{1}{2}x\right) \geq c_1x^{-\alpha},
			\end{align}
		for some positive constant $c_1$, where we used the fact that $\lim_{x\to\infty} \mathbf{P}_0\left(S_{\mathbf{e}}<x, |\xi_{\mathbf{e}}|<\frac{1}{2}x\right)=1$.
		Now we consider the numerator. Using the fact that $u\leq 1$, we get that that for $\delta^{\prime}>0$ sufficiently small,
		\begin{align}\label{upp-u-2}
			&\mathbf{E}_0\left[1_{\left\{S_{\mathbf{e}}<x\right\}} u^{2}\left(x-\xi_{\mathbf{e}}\right)\right]\nonumber\\
			&=\mathbf{E}_0\left[1_{\left\{S_{\mathbf{e}}<x\right\}} u^2\left(x-\xi_{\mathbf{e}}\right) 1_{\left\{\xi_{\mathbf{e}}<\left(1-\delta^{\prime}\right) x\right\}}\right]+\mathbf{E}_0\left[1_{\left\{S_{\mathbf{e}}<x\right\}} u^2\left(x-\xi_{\mathbf{e}}\right) 1_{\left\{\left(1-\delta^{\prime}\right) x<\xi_{\mathbf{e}}<x\right\}}\right] \nonumber\\
			&\leq u(\delta^{\prime} x)\mathbf{E}_0\left[1_{\left\{S_{\mathbf{e}}<x\right\}} u\left(x-\xi_{\mathbf{e}}\right)\right]+\mathbf{P}_0\left((1-\delta^{\prime})x<\xi_{\mathbf{e}}<x\right).
		\end{align}
	   Combining \eqref{de} and \eqref{upp-u-2},  we obtain that
		\begin{align}
			&\limsup_{x\to\infty}
			\frac{\mathbf{E}_0\left[1_{\left\{S_{\mathbf{e}}<x\right\}} u^{2}\left(x-\xi_{\mathbf{e}}\right)\right]}{\mathbf{E}_0\left[1_{\left\{S_{\mathbf{e}}<x\right\}} u\left(x-\xi_{\mathbf{e}}\right)\right]} \leq \lim_{x\to\infty} u(\delta^{\prime} x) + c_1\limsup_{x\to\infty}x^\alpha \mathbf{P}_0\left((1-\delta^{\prime})x<\xi_{\mathbf{e}}<x\right)\nonumber\\
			& = c_1\limsup_{x\to\infty}x^\alpha \int_0^\infty e^{-z}  \mathbf{P}_0\left((1-\delta^{\prime})x<\xi_{z}<x\right)\mathrm{d}z.
		\end{align}
		From \cite[Lemma 2.2]{Ren_Song_Zhang_2023}, when $x$ is large enough,
		\[
		x^\alpha \mathbf{P}_0\left((1-\delta^{\prime})x<\xi_{z}<x\right)\leq x^\alpha\mathbf{P}_0\left(|\xi_{z}| >(1-\delta^{\prime})x\right)
			\leq
		 c_2z, \quad \mbox{ for all } z>0
		\]
		for some positive constant $c_2$.
		Therefore, combining the dominated convergence theorem and \eqref{Tail-of-xi},
	   we get that
			\begin{align}
			&\limsup_{x\to\infty}
			\frac{\mathbf{E}_0\left[1_{\left\{S_{\mathbf{e}}<x\right\}} u^{2}\left(x-\xi_{\mathbf{e}}\right)\right]}{\mathbf{E}_0\left[1_{\left\{S_{\mathbf{e}}<x\right\}} u\left(x-\xi_{\mathbf{e}}\right)\right]} \leq  \frac{c_2c_+}{\alpha} \int_0^\infty ze^{-z}  \mathrm{d}z\frac{\delta'}{(1-\delta')^{1+\alpha}} \stackrel{\delta'\downarrow 0}{\longrightarrow}0,
		\end{align}
		which implies  \eqref{suffice}.		
		\end{proof}

		\section{Proof of Theorem \ref{t1}}

		\subsection{The case $0<\alpha <1$}

	Define
		\begin{equation}\label{def-eta}
			\eta_\alpha(\lambda):=\Gamma(1-\alpha) \lambda^{\alpha-1}
		\end{equation}
		and $\xi_{\mathbf{e}}^{+}=\max \left(0, \xi_{\mathbf{e}}\right)$.
			It follows from \cite[(2.1)]{profeta}  that
		\begin{equation}\label{asymptotics}
			\lim_{\lambda \downarrow 0}\frac{1-\mathbf{E}_0\left[e^{-\lambda S_{\mathbf{e}}}\right]}{\lambda\cdot \eta_\alpha(\lambda)} =\lim_{\lambda \downarrow 0} \frac{1-\mathbf{E}_0\left[e^{-\lambda \xi_{\mathbf{e}}^{+}}\right]}{\lambda\cdot \eta_\alpha(\lambda)} =
			\frac{c_+}{\alpha}.
		\end{equation}
		We will  use $\mathcal{L}[f]$ to denote the Laplace transform of a positive function $f$:
		$$
		\mathcal{L}[f](\lambda):=\int_{0}^{+\infty}e^{-\lambda x}f(x) \mathrm{d} x, \quad \lambda>0.
		$$

		The following lemma is given in \cite[Lemma 2.1]{profeta}
		\begin{lemma}\label{Important-lemma}
			Assume that $\alpha < 1$
			and that $f: [0,\infty)\to [0,\infty)$ is a positive and decreasing function.
			
			\noindent
			(i)
			For any $\lambda>0$, it holds that
			\[
			\int_0^\infty e^{-\lambda x} \mathbf{E}_0\left[1_{\{S_{\mathbf{e}}<x\}}f(x-\xi_{\mathbf{e}})\right]\mathrm{d}x \leq \mathbf{E}_0\left[e^{-\lambda S_{\mathbf{e}}}\right]\mathcal{L}[f](\lambda).
			\]
			
			\noindent
			(ii)
			For any $\lambda>0$, it holds that
			\begin{align}
				&	\int_0^\infty e^{-\lambda x} \mathbf{E}_0\left[1_{\{S_{\mathbf{e}}<x\}}f(x-L_{\mathbf{e}})\right]\mathrm{d}x\nonumber\\
				&\geq \mathbf{E}_0\left[e^{-\lambda \xi_{\mathbf{e}}^+}\right]\mathcal{L}[f](\lambda) +f(0)\frac{\mathbf{E}_0\left[e^{-\lambda S_{\mathbf{e}}}\right]-\mathbf{E}_0\left[e^{-\lambda \xi_{\mathbf{e}}^+}\right]}{\lambda} -  \mathbf{E}_0\left[1_{\left\{\xi_{\mathbf{e}}<0\right\}} \int_0^{-\xi_{\mathbf{e}}} e^{-\lambda z} f(z) \mathrm{d} z\right] .
			\end{align}
			
			\noindent
			(iii) If in addition that $\lim_{x\to\infty} f(x)=0$, then
			\begin{align}
				\lim_{\lambda\downarrow 0} \frac{1}{\eta_{\alpha}(\lambda)}  \mathbf{E}_0\left[1_{\left\{\xi_{\mathbf{e}}<0\right\}} \int_0^{-\xi_{\mathbf{e}}} e^{-\lambda z} f(z) \mathrm{d} z\right] = 0.
			\end{align}
		\end{lemma}
		The next lemma can be found in \cite[Lemma 2.2]{profeta}.

		\begin{lemma}\label{l3}
	For any $\lambda>0$, it holds that
			$$
			\frac{\lambda}{1-\mathbf{E}_0\left[e^{-\lambda S_{\mathbf{e}}}\right]} \mathcal{L}\left[\Phi_0+\Phi_R\right](\lambda) \leq 1
			$$
			and
			$$
			\frac{\lambda}{1-\mathbf{E}_0\left[e^{-\lambda \xi_{\mathbf{e}}^{+}}\right]} \mathcal{L}\left[\Phi_0+\Phi_R\right](\lambda) \geq 1-\lambda \mathcal{L}[u](\lambda)-\frac{\lambda}{1-\mathbf{E}_0\left[e^{-\lambda \xi_{\mathbf{e}}^{+}}\right]} \mathbf{E}_0\left[1_{\left\{\xi_{\mathbf{e}}<0\right\}} \int_0^{-\xi_{\mathbf{e}}} e^{-\lambda z} u(z) \mathrm{d} z\right] .
			$$
		\end{lemma}		
\bigskip

		\begin{proof}[Proof of Theorem \ref{t1} for $\alpha < 1$]
		Recall that $u(x)=\mathbb{P}\left(M \geq x\right)$.  Using a change of variables and the monotone convergence theorem, we get
		$$
		\lambda \mathcal{L}[u](\lambda)=\int_0^{+\infty} e^{-z} u\left(\frac{z}{\lambda}\right) \mathrm{d} z \underset{\lambda \downarrow 0}{\longrightarrow} 0 .
		$$
     Combining  \eqref{asymptotics} and Lemma \ref{Important-lemma} (iii) with $f=u$ and  Lemma \ref{l3}, we get
		\begin{equation}\label{upper bdd}
			\lim_{\lambda \downarrow 0}\frac{\mathcal{L}\left[\Phi_0+\Phi_R\right](\lambda)}{\eta_\alpha(\lambda)} =
			\frac{c_+}{\alpha}.
		\end{equation}

		We first consider the case $m=1$. In this case, $\Phi_0(x)=0$.
		Combining	\eqref{remainder1.2} and Lemma \ref{Important-lemma} (i) with $f=u^{\gamma}$ and $f=u^{\gamma+1}$, we get
		\begin{equation}\label{PR}
			\mathcal{L}\left[\Phi_R\right](\lambda) \leq
			C_2(\gamma)
			(1+\varepsilon)\mathcal{L}\left[u^{\gamma}\right](\lambda)+\frac{1}{\delta^{\gamma+1}} \mathcal{L}\left[u^{\gamma+1}\right](\lambda).
		\end{equation}
		Since $\lim _{x \rightarrow+\infty} u(x)=0$, for any
		$\varepsilon_1>0$,  there exists $A_1>0$ such that $u(x) \leq \varepsilon_1$ for $x \geq A_1$.
		Hence,
			\begin{align}\label{u^{gamma+1}}
				\mathcal{L}\left[u^{\gamma+1}\right](\lambda)&=\int_{0}^{A_1} e^{-\lambda x}u^{\gamma+1}(x)\mathrm{d} x+\int_{A_1}^{+\infty} e^{-\lambda x}u^{\gamma+1}(x)\mathrm{d} x\nonumber \\
				&\leq A_1
				+\varepsilon_1\int_{A_1}^{+\infty}e^{-\lambda x}u^{\gamma}(x)\mathrm{d} x\leq A_{1}+\varepsilon_1 \mathcal{L}\left[u^{\gamma}\right](\lambda),
			\end{align}
		where in the first inequality we used the fact that  $u \leq 1$.
		 Thus,
		 combining
		  \eqref{upper bdd}, \eqref{PR} and \eqref{u^{gamma+1}}, we have
		\begin{align*}
			\frac{c_+}{\alpha}=\lim_{\lambda \downarrow 0}\frac{\mathcal{L}\left[\Phi_R\right](\lambda)}{\eta_\alpha(\lambda)}
			\leq & \liminf _{\lambda \downarrow 0}
			\frac{\left(C_2(\gamma)(1+\varepsilon)+\varepsilon_1
				/\delta^{\gamma+1}\right)\mathcal{L}\left[u^{\gamma}\right](\lambda)+A_1
				/\delta^{\gamma+1}}{\eta_\alpha(\lambda)}\\
			=& \liminf _{\lambda \downarrow 0}
			\frac{\left(C_2(\gamma)(1+\varepsilon)+\varepsilon_1
				/\delta^{\gamma+1}\right)\mathcal{L}\left[u^{\gamma}\right](\lambda)}{\eta_\alpha(\lambda)}.
		\end{align*}
		Letting $\varepsilon_1\to 0$ first and then $\varepsilon\to 0$, we get that
		\begin{align} \label{proof-liminf}
	\frac{c_+}{\alpha}\leq \liminf _{\lambda \downarrow 0} \frac{C_2(\gamma)\mathcal{L}\left[u^{\gamma}\right](\lambda)}{\eta_\alpha(\lambda)}.
		\end{align}

		On the other hand, combining \eqref{remainder1.1}, \eqref{u^{gamma+1}},
Lemma \ref{Important-lemma} (ii) with $f=u^{\gamma}$ and Lemma \ref{Important-lemma} (i) with $f=u^{\gamma+1}$, we see that
			\begin{align}\label{3}
				& \mathcal{L}\left[\Phi_R\right](\lambda) \nonumber\\
				&\geq
				C_2(\gamma)(1-\varepsilon)\left( \mathbf{E}_0\left[e^{-\lambda \xi_{\mathbf{e}}^{+}}\right] \mathcal{L}[u^{\gamma}](\lambda)+ \frac{\mathbf{E}_0\left[e^{-\lambda S_{\mathbf{e}}}\right]-\mathbf{E}_0\left[e^{-\lambda \xi_{\mathbf{e}}^{+}}\right]}{\lambda}-\mathbf{E}_0\left[1_{\left\{\xi_{\mathbf{e}}<0\right\}} \int_0^{-\xi_{\mathbf{e}}} e^{-\lambda z} u^{\gamma}(z) \mathrm{d} z\right]\right)\nonumber\\
				&-\frac{C_2(\gamma)(1-\varepsilon)}{\delta}\left( A_1+\varepsilon_1 \mathcal{L}\left[u^{\gamma}\right](\lambda)\right).
			\end{align}
		Dividing both sides by $\eta_\alpha(\lambda)$ and
          using Lemma \ref{Important-lemma} (iii) with $f=u^{\gamma}$, we obtain
		$$
		\frac{c_+}{\alpha} \geq \limsup _{\lambda \downarrow 0} \frac{C_2(\gamma)(1-\varepsilon)\left(1-\varepsilon_1
			/\delta\right)\mathcal{L}[u^{\gamma}](\lambda)}{\eta_\alpha(\lambda)}.
		$$
	Letting $\varepsilon_1\to 0$ first and then $\varepsilon\to0$, we conclude that
		\begin{align} \label{proof-limsup}
		\frac{c_+}{\alpha} \geq \limsup _{\lambda \downarrow 0} \frac{C_2(\gamma)\mathcal{L}\left[u^{\gamma}\right](\lambda)}{\eta_\alpha(\lambda)}.
		\end{align}
		Combining \eqref{proof-liminf} and \eqref{proof-limsup}, we conclude that
		$$
		\lim_{\lambda \downarrow 0}
			\frac{\mathcal{L}\left[u^{\gamma}\right](\lambda)}{\eta_\alpha(\lambda)}=\frac{c_+}{\alpha C_2(\gamma)}.
		$$
		Hence, by the Tauberian theorem, the above limit is equivalent to
		$$
		\lim_{x \to +\infty}
				\frac{1}{\eta_\alpha\left(\frac{1}{x}\right)}\int_0^{x}u^{\gamma}(z) \mathrm{d} z=	\frac{c_+}{\alpha\Gamma(2-\alpha)C_2(\gamma)}.
		$$
Applying  Karamata's
monotone density theorem \cite[Theorem 1.7.2]{bingham_goldie_teugels_1987}, we get the desired result for $m=1$.

		Now we deal with the subcritical case $m<1$. For any  $\varepsilon^{\prime}>0$,
				by \eqref{remainder2},  we see that there exists a constant $A'$ such that  for all $x\geq A'$,
		\begin{equation}\label{varepsilon_prime}
		 0\leq 	\Phi_R(x)\leq \varepsilon^{\prime}\Phi_0(x).
		\end{equation}
		Similar to \eqref{u^{gamma+1}}, using  \eqref{varepsilon_prime},  we get that for all $\lambda>0$,
		\begin{align}\label{Bounds-for-Phi-R}
			0\leq \mathcal{L}{[\Phi_R]}(\lambda) \leq A'+ \varepsilon^{\prime} \int_{A'}^\infty e^{-\lambda x} \Phi_0(x)\mathrm{d}x \leq A'+  \varepsilon^{\prime} \mathcal{L}{[\Phi_0]}(\lambda),
		\end{align}
		which together with \eqref{upper bdd} implies that
		\begin{align}
		 & 	\limsup_{\lambda\downarrow 0} \frac{\mathcal{L}[\Phi_0](\lambda)}{\eta_\alpha(\lambda)} \leq  \frac{c_+}{\alpha}\leq \liminf_{\lambda\downarrow 0} \frac{(1+\varepsilon')\mathcal{L}[\Phi_0](\lambda)+ A'}{\eta_\alpha(\lambda)} =  (1+\varepsilon')\liminf_{\lambda\downarrow 0} \frac{\mathcal{L}[\Phi_0](\lambda)}{\eta_\alpha(\lambda)} .
		\end{align}
		Letting $\varepsilon'\downarrow 0$, we get
		$$
		\lim_{\lambda \downarrow 0}
		\frac{1}{\eta_\alpha(\lambda)}\mathcal{L}[u](\lambda)=\frac{c_+}{\alpha(1-m)}.
		$$
		Hence, by the Tauberian theorem, we have
		$$
		\lim_{x \rightarrow +\infty}
		\frac{1}{\eta_\alpha\left(\frac{1}{x}\right)}		\int_0^{x}u(z)  \mathrm{d} z =\frac{c_+}{\alpha\Gamma(2-\alpha)(1-m)} .
		$$
  Applying Karamata's
monotone density theorem \cite[Theorem 1.7.2]{bingham_goldie_teugels_1987}, we get the desired result.
		\end{proof}

\subsection{Proof of Theorem \ref{t1} for $1\leq \alpha<2$}
		 It follows from  \cite[(3.3)] {profeta} that for $\alpha\in (1,2)$,
		\begin{equation}\label{a2}
			\lim_{\lambda \downarrow 0}\frac{\int_0^{+\infty} e^{-\lambda x} x \mathbf{P}_0\left(S_{\mathbf{e}} \geq x\right)  \mathrm{d} x}{\lambda^{\alpha-2}}
			=\lim_{\lambda \downarrow 0}\frac{1-\mathbf{E}_0\left[e^{-\lambda S_{\mathbf{e}}}\right]-\lambda \mathbf{E}_0\left[S_{\mathbf{e}} e^{-\lambda S_{\mathbf{e}}}\right]}{\lambda^2\cdot\lambda^{\alpha-2}} = \frac{c_+\Gamma(2-\alpha)}{\alpha}.
		\end{equation}
		For $\alpha=1$, combining \eqref{Tail-of-xi}, \cite[Lemma 2.2]{Ren_Song_Zhang_2023} and the dominated convergence theorem,  \eqref{a2} remains true for $\alpha=1$ since $\Gamma(1)=1$.  Moreover, \eqref{a2} also holds with $S_{\mathbf{e}}$ replaced by $\xi_{\mathbf{e}}^+$.
		\begin{lemma}\label{Important-lemma2}
			Assume that $\alpha\in [1,2)$
			 and that $f:[0,\infty)\to [0,\infty)$ is a positive and decreasing function.
			
			(i) We have the following upper bound
			\[
				\int_0^\infty e^{-\lambda x}x \mathbf{E}_0\left[1_{\{S_{\mathbf{e}}<x\}}f(x-\xi_{\mathbf{e}})\right]\mathrm{d}x \leq
				\mathbf{E}_0\left[e^{-\lambda S_{\mathbf{e}}}\right]\mathcal{L}[xf(x)](\lambda)+\mathbf{E}_0\left[S_{\mathbf{e}}e^{-\lambda S_{\mathbf{e}}}\right]\mathcal{L}[f](\lambda).
			\]

        	(ii) We have the following lower bound
        \begin{align}
        	&	\int_0^\infty e^{-\lambda x}x \mathbf{E}_0\left[1_{\{S_{\mathbf{e}}<x\}}f(x-\xi_{\mathbf{e}})\right]\mathrm{d}x \nonumber\\
        	&\geq f(0)\int_0^{+\infty} e^{-\lambda x} x\left(\mathbf{P}_0\left(\xi_{\mathbf{e}} \geq x\right)-\mathbf{P}_0\left(S_{\mathbf{e}} \geq x\right)\right) \mathrm{d} x + \mathbf{E}_0\left[e^{-\lambda \xi_{\mathbf{e}}}1_{\{\xi_{\mathbf{e}}\geq 0 \}}\right]\mathcal{L}[xf(x)](\lambda)\nonumber\\
        	&\qquad +\int_0^\infty e^{-\lambda x}x \mathbf{E}_0\left[1_{\{ \xi_{\mathbf{e}}<0\}}f(x-\xi_{\mathbf{e}})\right]\mathrm{d}x.
        \end{align}
		\end{lemma}
		\begin{proof}
			For (i), see \cite[Lemma 3.1]{profeta} for the proof of $\alpha\in (1,2)$ and the proof for $\alpha	=1$ is the same.
			Now we prove (ii). Combining the inequalities $1_{\{S_{\mathbf{e}}\geq x \}} \geq 1_{\{\xi_{\mathbf{e}}\geq x \}} $  and $f(x)\leq f(0)$ for all $x\geq 0$, we have
			\begin{align}\label{eq1}
			&	\int_0^\infty e^{-\lambda x}x \mathbf{E}_0\left[1_{\{S_{\mathbf{e}}<x\}}f(x-\xi_{\mathbf{e}})\right]\mathrm{d}x \nonumber\\
			&\geq f(0)\int_0^{+\infty} e^{-\lambda x} x\left(\mathbf{P}_0\left(\xi_{\mathbf{e}} \geq x\right)-\mathbf{P}_0\left(S_{\mathbf{e}} \geq x\right)\right) \mathrm{d} x +	\int_0^\infty e^{-\lambda x}x \mathbf{E}_0\left[1_{\{\xi_{\mathbf{e}}<x\}}f(x-\xi_{\mathbf{e}})\right]\mathrm{d}x .
			\end{align}
By Fubini's theorem, we have
			\begin{align}\label{eq2}
				&	\int_0^\infty e^{-\lambda x}x \mathbf{E}_0\left[1_{\{0\leq \xi_{\mathbf{e}}<x\}}f(x-\xi_{\mathbf{e}})\right]\mathrm{d}x = \int_0^\infty e^{-\lambda x}x \int_0^x f(x-y)\mathbf{P}_0(\xi_{\mathbf{e}}\in \mathrm{d} y)\mathrm{d}x \nonumber\\
				& =  \int_0^\infty  \mathbf{P}_0(\xi_{\mathbf{e}}\in \mathrm{d} y) \int_y^\infty  e^{-\lambda x}x f(x-y)\mathrm{d}x \geq \int_0^\infty  \mathbf{P}_0(\xi_{\mathbf{e}}\in \mathrm{d} y) \int_y^\infty  e^{-\lambda x}(x-y) f(x-y)\mathrm{d}x \nonumber\\
				& =\mathbf{E}_0\left[e^{-\lambda \xi_{\mathbf{e}}}1_{\{\xi_{\mathbf{e}}\geq 0 \}}\right]\mathcal{L}[xf(x)](\lambda).
			\end{align}
			Now (ii) follows from \eqref{eq1} and \eqref{eq2}
		\end{proof}

		\begin{lemma}\label{l5}
			For any $\lambda>0$, it holds that

(i)
			$$
			\frac{\lambda^2}{1-\mathbf{E}_0\left[e^{-\lambda S_{\mathbf{e}}}\right]-\lambda \mathbf{E}_0\left[S_{\mathbf{e}} e^{-\lambda S_{\mathbf{e}}}\right]} \mathcal{L}\left[x(\Phi_0(x)+\Phi_R(x))\right](\lambda) \leq 1+\frac{\lambda^2 \mathbf{E}_0\left[S_{\mathbf{e}} e^{-\lambda S_{\mathbf{e}}}\right]}{1-\mathbf{E}_0\left[e^{-\lambda S_{\mathbf{e}}}\right]-\lambda \mathbf{E}_0\left[S_{\mathbf{e}} e^{-\lambda S_{\mathbf{e}}}\right]} \mathcal{L}[u](\lambda),
			$$

(ii)
		\begin{align}
				&\frac{\lambda^2}{1-\mathbf{E}_0\left[e^{-\lambda \xi_{\mathbf{e}}^{+}}\right]-\lambda \mathbf{E}_0\left[\xi_{\mathbf{e}}^{+} e^{-\lambda \xi_{\mathbf{e}}^{+}}\right]} \mathcal{L}\left[x(\Phi_0(x)+\Phi_R(x))\right](\lambda) \nonumber\\
			& \geq 1- \frac{\lambda^2 \mathbf{E}_0\left[ 1- e^{-\lambda \xi_{\mathbf{e}}^+}\right]\mathcal{L}[xu(x)](\lambda)}{1-\mathbf{E}_0\left[e^{-\lambda \xi_{\mathbf{e}}^{+}}\right]-\lambda \mathbf{E}_0\left[\xi_{\mathbf{e}}^{+} e^{-\lambda \xi_{\mathbf{e}}^{+}}\right]}  \nonumber\\
			& \quad - \frac{\lambda^2\left( \mathbf{P}_0 \left(\xi_{\mathbf{e}}\leq 0\right) \mathcal{L}(xu(x))(\lambda)
			-
			\int_0^\infty e^{-\lambda x}x \mathbf{E}_0\left[1_{\{ \xi_{\mathbf{e}}<0\}}u(x-\xi_{\mathbf{e}})\right]\mathrm{d}x\right)}{1-\mathbf{E}_0\left[e^{-\lambda \xi_{\mathbf{e}}^{+}}\right]-\lambda \mathbf{E}_0\left[\xi_{\mathbf{e}}^{+} e^{-\lambda \xi_{\mathbf{e}}^{+}}\right]}.
		\end{align}
		\end{lemma}
    \begin{proof}
    	The proof of (i) for $\alpha\in (1,2)$ can be found in  \cite[Lemma 3.2]{profeta} and the case $\alpha=1$ can be treated similarly. Now we prove (ii). Combining
	\eqref{integral equation}
	and Lemma \ref{Important-lemma2} (ii) with $f=u$, we see
	that
    	\begin{align}
    		& \mathcal{L}\left[x(\Phi_0(x)+\Phi_R(x))\right](\lambda)  + \mathcal{L}(xu(x))(\lambda) \nonumber\\
    		& = \int_0^{+\infty} e^{-\lambda x} x \mathbf{P}_0\left(S_{\mathbf{e}} \geq x\right)  \mathrm{d} x+ \int_0^\infty e^{-\lambda x}x \mathbf{E}_0\left[1_{\{S_{\mathbf{e}}<x\}}u(x-\xi_{\mathbf{e}})\right]\mathrm{d}x \nonumber\\
    		& \geq \int_0^{+\infty} e^{-\lambda x} x \mathbf{P}_0\left(\xi_{\mathbf{e}} \geq x\right)  \mathrm{d} x+  \mathbf{E}_0\left[e^{-\lambda \xi_{\mathbf{e}}}1_{\{\xi_{\mathbf{e}}\geq 0 \}}\right]\mathcal{L}[xu(x)](\lambda)\nonumber\\
    		&\qquad +
    		\int_0^\infty e^{-\lambda x}x \mathbf{E}_0\left[1_{\{\xi_{\mathbf{e}}<x\}}f(x-\xi_{\mathbf{e}})\right]\mathrm{d}x .
    	\end{align}
   From the argument of \eqref{a2} and
    the fact that $\xi_{\mathbf{e}}= \xi_{\mathbf{e}}^+$ on the set $\{\xi_{\mathbf{e}}\geq x\}$ for $x\geq 0$, we get the lower bound
    	\begin{align}
    	&	\mathcal{L}\left[x(\Phi_0(x)+\Phi_R(x))\right](\lambda)  + \mathcal{L}(xu(x))(\lambda)\nonumber\\
    	&\geq \int_0^{+\infty} e^{-\lambda x} x \mathbf{P}_0\left(\xi_{\mathbf{e}}^+  \geq x\right)  \mathrm{d} x-\mathbf{E}_0\left[(1-e^{-\lambda\xi_{\mathbf{e}}^+})1_{\left\{\xi_{\mathbf{e}}\geq 0\right\}}\right]\mathcal{L}(xu(x))(\lambda) \nonumber\\
    	&\quad - \left( \mathbf{P}_0 \left(\xi_{\mathbf{e}}\leq 0\right) \mathcal{L}(xu(x))(\lambda)
    	-   \int_0^\infty e^{-\lambda x}x \mathbf{E}_0\left[1_{\{ \xi_{\mathbf{e}}<0\}}u(x-\xi_{\mathbf{e}})\right]\mathrm{d}x\right)\nonumber\\
    	& = \frac{1-\mathbf{E}_0\left[e^{-\lambda \xi_{\mathbf{e}}^{+}}\right]-\lambda \mathbf{E}_0\left[\xi_{\mathbf{e}}^{+} e^{-\lambda \xi_{\mathbf{e}}^{+}}\right]}{\lambda^2} -\mathbf{E}_0\left[1-e^{-\lambda\xi_{\mathbf{e}}^+}\right]\mathcal{L}(xu(x))(\lambda) \nonumber\\
    	&\quad - \left( \mathbf{P}_0 \left(\xi_{\mathbf{e}}\leq 0\right) \mathcal{L}(xu(x))(\lambda)
    	-   \int_0^\infty e^{-\lambda x}x \mathbf{E}_0\left[1_{\{ \xi_{\mathbf{e}}<0\}}u(x-\xi_{\mathbf{e}})\right]\mathrm{d}x\right),
    	\end{align}
    	which implies the desired result.
    \end{proof}

       In the critical case $m=1$, we will need the following
     a priori upper bound of $u$.
		\begin{lemma}\label{l6}
			If $\alpha\in (1,2)$ and $m=1$,
			then there exists a constant $A>0$ such that,
			$$
			u(x)\leq A x^{-\alpha / \gamma}, \qquad \mbox{ for all } x\geq 0.
			$$
		\end{lemma}
		\begin{proof}
                The proof is similar to that of \cite[Lemma 3.3]{profeta}.
			The main difference is that  we have  \eqref{222} for our branching mechanism. Denote by $\underline{M}^{(t)}$ the maximum of $(Z_s: s\geq 0)$ on $[0,t]$ and by $\overline{M}^{(t)}$ the maximum of $(Z_s: s\geq 0)$ on $[t,+\infty]$.
			Since  $\alpha \in (1,2)$, we have $\rho:=\mathbf{P}_0(\xi_1 \geq 0)\in [1-\frac{1}{\alpha},\frac{1}{\alpha}]$.
			It follows that
			\begin{align}\label{domi-u}
				u(x)&=\mathbb P(M\geq x)
				\leq \mathbb P(\underline{M}^{(t)} \geq x)+\mathbb P(\overline{M}^{(t)} \geq x)\nonumber\\
				&\leq \mathbb P(\underline{M}^{(t)} \geq x)+\mathbb P\left(N_{t} \geq 1\right).
			\end{align}
			Define $T_x= \inf \left\{t \geq 0: \exists u \in N_{t}, X_{u}(t)\geq x\right\}$. Then $\left\{T_x \leq t\right\}=\left\{\underline{M}^{(t)} \geq x\right\}$. Applying strong Markov property at $T_x$, we get
			$$
			\mathbb{E}\left[
			\sum_{u\in N_t} 1_{\{X_u(t)\geq x\}}
			\mid \underline{M}^{(t)} \geq x\right] \geq \rho,
			$$
			which implies that
			\begin{equation}\label{111}
				\mathbb{P}\left(\underline{M}^{(t)} \geq x\right) \leq \rho^{-1} \mathbb{E}\left[
				\sum_{u\in N_t} 1_{\{X_u(t)\geq x\}}
				\right] \leq \rho^{-1}\mathbb{E}[N_{t}] \mathbf{P}_0\left(\xi_t \geq x\right)=\rho^{-1} \mathbf{P}_0\left(t^{1 / \alpha} \xi_1 \geq x\right),
			\end{equation}
			where we used the facts that $N_t$ is independent of the spatial positions, and $\mathbb{E}[N_{t}]=1$ for all $t \geq 0$.
			It follows from \eqref{generating} that the function
						$\sum_{k=0}^\infty p_k(1-x)^k-(1-x)$
			is a regularly varying at 0 with index $\gamma$.
			Hence by \cite{Zolotarev} [Theorem 4]  we know that
						$Q(t):=\mathbb P(N_{t} \geq 1)$ satisfies the following equation for some positive constant $c$:
			$$
			\lim_{t \to \infty}\frac{Q(t)}{t(\gamma-1) c Q(t)^{\gamma}}=1,
			$$
which implies that
					\begin{equation}\label{222}
				\mathbb P(N_{t} \geq 1) \sim \left(\frac{1}
							{c (\gamma-1)t}\right)^{\frac{1}{\gamma-1}}.
			\end{equation}
			Taking $t=x^{\alpha(1-\frac{1}{\gamma})}$ in \eqref{111} and \eqref{222},
			using \eqref{domi-u},  we get that there is some constant $A>0$ such that
			$$
			u(x) \leq A  x^{-\alpha / \gamma}.
			$$
			This completes the proof of the Lemma.
		\end{proof}

		\begin{proof}[Proof of Theorem \ref{t1} for $\alpha\in [1,2)$]
		We first prove that
		\begin{equation}\label{toprove-Laplace-limit}
			\lim_{\lambda \downarrow 0}\frac{\mathcal{L}\left[x (\Phi_0(x)+\Phi_R(x))\right](\lambda)}{\lambda^{\alpha-2}}
			=   \frac{c_+\Gamma(2-\alpha)}{\alpha} .
		\end{equation}
			For the upper bound, combining \eqref{a2} and Lemma \ref{l5}(i), we have
			\begin{align}\label{e4}
			& \frac{\alpha}{c_+\Gamma(2-\alpha)}	\limsup_{\lambda \downarrow 0}\frac{\mathcal{L}\left[x (\Phi_0(x)+\Phi_R(x))\right](\lambda)}{\lambda^{\alpha-2}}  \leq 1+ \limsup_{\lambda \downarrow 0} \frac{\lambda^2 \mathbf{E}_0\left[S_{\mathbf{e}} e^{-\lambda S_{\mathbf{e}}}\right]}{1-\mathbf{E}_0\left[e^{-\lambda S_{\mathbf{e}}}\right]-\lambda \mathbf{E}_0\left[S_{\mathbf{e}} e^{-\lambda S_{\mathbf{e}}}\right]} \mathcal{L}[u](\lambda)\nonumber\\
			& =  1+\frac{\alpha}{c_+\Gamma(2-\alpha)} \limsup_{\lambda \downarrow 0} \lambda^{2-\alpha} \mathbf{E}_0\left[S_{\mathbf{e}} e^{-\lambda S_{\mathbf{e}}}\right]\mathcal{L}[u](\lambda).
			\end{align}
			When $\alpha\in (\gamma,2)$, combining Lemma \ref{l6} and the fact that $\mathbf{E}_0(S_{\mathbf{e}})<\infty$, we see that
			\begin{align}\label{eq4}
				\lambda^{2-\alpha} \mathbf{E}_0\left[S_{\mathbf{e}} e^{-\lambda S_{\mathbf{e}}}\right]\mathcal{L}[u](\lambda)\leq \mathbf{E}_0\left[S_{\mathbf{e}} \right] \lambda^{2-\alpha}  \left(1+ A\int_1^\infty  x^{-\alpha/\gamma}\mathrm{d}x\right)\stackrel{\lambda\downarrow 0}{\longrightarrow} 0.
			\end{align}
			When $\alpha\in (1, \gamma]$,
			using $\alpha\in (1,2)$ and $\gamma\in (1,2]$, we have $1+\alpha \gamma^{-1} = \alpha + \gamma^{-1} \left(1-(\alpha-1)(\gamma-1)\right)>\alpha$. Thus  there exists $\delta\in (0,1)$ such that $1+\alpha\delta /\gamma >\alpha$.
			Therefore, combining Lemma \ref{l6} and $\mathbf{E}_0(S_{\mathbf{e}})<\infty$,
				\begin{align}\label{eq5}
				&\lambda^{2-\alpha} \mathbf{E}_0\left[S_{\mathbf{e}} e^{-\lambda S_{\mathbf{e}}}\right]\mathcal{L}[u](\lambda)\leq A \mathbf{E}_0\left[S_{\mathbf{e}} \right] \lambda^{2-\alpha}   \int_0^\infty e^{-\lambda x} x^{-\alpha\delta/\gamma}\mathrm{d}x\nonumber\\
				&= A \mathbf{E}_0\left[S_{\mathbf{e}} \right] \lambda^{1+\alpha\delta/\gamma-\alpha}   \int_0^\infty e^{-x} x^{-\alpha\delta/\gamma}\mathrm{d}x\stackrel{\lambda\downarrow0}{\longrightarrow}0.
			\end{align}
			When $\alpha=1$,
	fix a constant   $\delta\in (1-1/\gamma,1)$.	
 Combining  \cite[(2.1)]{profeta} and Markov's inequality, we have
			\[
			y^{\delta}\mathbf{P}_0(S_{\mathbf{e}}\geq y) \leq
\frac{y^{\delta}}{(1-e^{-1})}
\mathbf{E}_0(1-e^{-y^{-1}S_{\mathbf{e}}}) =\frac{y^{\delta-1} \ln y}{(1-e^{-1})} \frac{\mathbf{E}_0(1-e^{-y^{-1}S_{\mathbf{e}}})}{y^{-1}\ln y} \stackrel{y\to\infty}{\longrightarrow} 0.
			\]		
Thus there exists $c_\delta>0$ such that
  $\mathbf{P}_0(S_{\mathbf{e}}\geq y)\leq c_\delta y^{-\delta}$ for any $y>0$,
which implies that
			\begin{align}\label{eq6}
					\lambda \mathbf{E}_0\left[S_{\mathbf{e}} e^{-\lambda S_{\mathbf{e}}}\right]\mathcal{L}[u](\lambda) &  =  	
 \lambda\int_0^\infty (1-\lambda y)e^{-\lambda y} \mathbf{P}_0(S_{\mathbf{e}}>y)\mathrm{d}y \int_0^\infty e^{-\lambda x} u(x)\mathrm{d}x \nonumber\\&
					 \leq	A \lambda \int_0^\infty e^{-\lambda y}\mathbf{P}_0(S_{\mathbf{e}}\geq y)\mathrm{d}y \int_0^\infty e^{-\lambda x} x^{-1/\gamma}\mathrm{d} x\nonumber\\
					& \leq Ac_\delta  \lambda \int_0^\infty e^{-\lambda y} y^{-\delta}\mathrm{d}y \int_0^\infty e^{-\lambda x} x^{-1/\gamma}\mathrm{d} x\nonumber\\
					& = Ac_\delta  \lambda^{\delta+ \gamma^{-1}-1} \int_0^\infty e^{- y} y^{-\delta}\mathrm{d}y \int_0^\infty e^{- x} x^{-1/\gamma}\mathrm{d} x\stackrel{\lambda\downarrow 0}{\longrightarrow}0.
			\end{align}
			Combining \eqref{e4},  \eqref{eq4}, \eqref{eq5} and \eqref{eq6}, we get
			\begin{align}\label{to-prove-upper}
				\limsup_{\lambda \downarrow 0}\frac{\mathcal{L}\left[x (\Phi_0(x)+\Phi_R(x))\right](\lambda)}{\lambda^{\alpha-2}}
				\leq   \frac{c_+\Gamma(2-\alpha)}{\alpha} .
			\end{align}
			
			Now we prove the lower bound. Similarly, combining \eqref{a2} and Lemma \ref{l5}(ii), we see that for any $\varepsilon>0$,
			\begin{align}\label{eq8}
				& \frac{\alpha}{c_+\Gamma(2-\alpha)}	\liminf_{\lambda \downarrow 0}\frac{\mathcal{L}\left[x (\Phi_0(x)+\Phi_R(x))\right](\lambda)}{\lambda^{\alpha-2}}
				\nonumber\\
				& \geq  1- \frac{\alpha}{c_+\Gamma(2-\alpha)} \limsup_{\lambda \downarrow 0} \lambda^{2-\alpha} \mathbf{E}_0\left[ 1- e^{-\lambda \xi_{\mathbf{e}}^+}\right]\mathcal{L}[xu(x)](\lambda)  - \frac{\alpha}{c_+\Gamma(2-\alpha)} \nonumber\\
				&\quad \times \limsup_{\lambda \downarrow 0} \lambda^{2-\alpha} \left( \mathbf{P}_0 \left(\xi_{\mathbf{e}}\leq 0\right) \mathcal{L}(xu(x))(\lambda) - \int_0^\infty e^{-\lambda x}x \mathbf{E}_0\left[1_{\{ \xi_{\mathbf{e}}<0\}}u(x-\xi_{\mathbf{e}})\right]\mathrm{d}x\right).
			\end{align}
			From Lemma \ref{l6},  it holds that
			\begin{align}\label{eq10}
				\mathcal{L}[xu(x)](\lambda) \leq A\int_0^\infty e^{-\lambda x} x^{1- \gamma^{-1}\alpha}\mathrm{d}x = A \lambda^{\alpha/\gamma -2}\int_0^\infty e^{- x} x^{1- \gamma^{-1}\alpha}\mathrm{d}x.
			\end{align}
			Since $1-\gamma^{-1}\alpha > 1- 2=-1$, we conclude from the above inequality that for any $\alpha\in [1,2)$, there exists $A'$ such that
			\begin{align}
				\lambda^{2-\alpha} \mathbf{E}_0\left[ 1- e^{-\lambda \xi_{\mathbf{e}}^+}\right]\mathcal{L}[xu(x)](\lambda)  \leq 	A' \lambda^{\alpha/\gamma-\alpha} \mathbf{E}_0\left[ 1- e^{-\lambda \xi_{\mathbf{e}}^+}\right].
			\end{align}
Note that when $\alpha=1$, by \cite[(2.1)]{profeta},   $\lim_{\lambda\downarrow 0} \frac{1}{\lambda \ln(\lambda^{-1})} \mathbf{E}_0\left[ 1- e^{-\lambda \xi_{\mathbf{e}}^+}\right] = \frac{c_+}{\alpha}$, and   when $\alpha\in (1,2)$, $\lim_{\lambda\downarrow 0} \frac{1}{\lambda } \mathbf{E}_0\left[ 1- e^{-\lambda \xi_{\mathbf{e}}^+}\right] =\mathbf{E}_0(\xi_{\mathbf{e}}^+ ) $.  Using the fact that $\alpha/\gamma -\alpha +1 = \gamma^{-1}(1-(\alpha-1)(\gamma-1))>0$, we have
\begin{align}\label{eq7}
			\limsup_{\lambda\downarrow 0}	\lambda^{2-\alpha} \mathbf{E}_0\left[ 1- e^{-\lambda \xi_{\mathbf{e}}^+}\right]\mathcal{L}[xu(x)](\lambda) =0.
			\end{align}
 Plugging \eqref{eq7} into \eqref{eq8} yields that
				\begin{align}\label{eq9}
				& \frac{\alpha}{c_+\Gamma(2-\alpha)}	\liminf_{\lambda \downarrow 0}\frac{\mathcal{L}\left[x (\Phi_0(x)+\Phi_R(x))\right](\lambda)}{\lambda^{\alpha-2}}
				 \geq  1 - \frac{\alpha}{c_+\Gamma(2-\alpha)} \nonumber\\
				&\quad \times \limsup_{\lambda \downarrow 0} \lambda^{2-\alpha} \left( \mathbf{P}_0 \left(\xi_{\mathbf{e}}\leq 0\right) \mathcal{L}(xu(x))(\lambda) -  \int_0^\infty e^{-\lambda x}x \mathbf{E}_0\left[1_{\{ \xi_{\mathbf{e}}<0\}}u(x-\xi_{\mathbf{e}})\right]\mathrm{d}x\right).
			\end{align}

When $\alpha\in (1,2)$, using the fact that $\mathbf{E}_0(|\xi_{\mathbf{e}}|)<\infty$, we see that
			\begin{align}\label{eq11}
				& \int_0^\infty e^{-\lambda x}x \mathbf{E}_0\left[
    1_{\{ \xi_{\mathbf{e}}\leq 0\}}
u(x-\xi_{\mathbf{e}})\right]\mathrm{d}x = \mathbf{E}_0\left(
1_{\left\{\xi_{\mathbf{e}}\leq 0\right\}}
\int_{-\xi_{\mathbf{e}}}^\infty e^{-\lambda(x+\xi_{\mathbf{e}})}(x+\xi_{\mathbf{e}})u(x)\mathrm{d}x\right)\nonumber\\
				&\geq \mathbf{E}_0\left(
 1_{\left\{\xi_{\mathbf{e}}\leq 0\right\}}
\int_{-\xi_{\mathbf{e}}}^\infty e^{-\lambda x}(x+\xi_{\mathbf{e}})u(x)\mathrm{d}x\right) \geq  \mathbf{P}_0 \left(\xi_{\mathbf{e}}\leq 0\right) \mathcal{L}(xu(x))(\lambda) \nonumber\\
				&\quad - \mathbf{E}_0\left(1_{\left\{
\xi_{\mathbf{e}}\leq 0\right\}}
\int_0^{-\xi_{\mathbf{e}}} e^{-\lambda x}xu(x)\mathrm{d}x\right) - \mathbf{E}_0(|\xi_{\mathbf{e}}|) \mathcal{L}(u)(\lambda)\nonumber\\
				&\geq \mathbf{P}_0 \left(\xi_{\mathbf{e}}\leq 0\right) \mathcal{L}(xu(x))(\lambda) -2 \mathbf{E}_0(|\xi_{\mathbf{e}}|) \mathcal{L}(u)(\lambda).
			\end{align}
Therefore, combining \eqref{eq4} and \eqref{eq5},
 we conclude from \eqref{eq11} that
		   \begin{align}
		    &\limsup_{\lambda \downarrow 0} \lambda^{2-\alpha} \left( \mathbf{P}_0 \left(\xi_{\mathbf{e}}\leq 0\right) \mathcal{L}(xu(x))(\lambda) -  \int_0^\infty e^{-\lambda x}x \mathbf{E}_0\left[1_{\{ \xi_{\mathbf{e}}<0\}}u(x-\xi_{\mathbf{e}})\right]\mathrm{d}x\right)\nonumber\\
		    &\leq 2\mathbf{E}_0(|\xi_{\mathbf{e}}|)  \limsup_{\lambda \downarrow 0} \lambda^{2-\alpha}  \mathcal{L}(u)(\lambda) =
		    0.
		   \end{align}
Combining the above inequality with
		   \eqref{to-prove-upper} and \eqref{eq9}, we get \eqref{toprove-Laplace-limit} holds for $\alpha\in (1,2)$.

 Now we consider the case $\alpha=1$.
For any fixed $\delta\in (0,1)$,
 noticing that $\mathbf{E}_0(|\xi_{\mathbf{e}}|^{1-\delta})<\infty$, we get that
		   	\begin{align}\label{eq11'}
		   	& \int_0^\infty e^{-\lambda x}x \mathbf{E}_0\left[1_{\{ \xi_{\mathbf{e}}\leq 0\}}u(x-\xi_{\mathbf{e}})\right]\mathrm{d}x = \mathbf{E}_0\left(1_{\left\{\xi_{\mathbf{e}}\leq 0\right\}}\int_{-\xi_{\mathbf{e}}}^\infty e^{-\lambda(x+\xi_{\mathbf{e}})}(x+\xi_{\mathbf{e}})u(x)\mathrm{d}x\right)\nonumber\\
		   	&\geq \mathbf{E}_0\left(1_{\left\{\xi_{\mathbf{e}}\leq 0\right\}}\int_{-\xi_{\mathbf{e}}}^\infty e^{-\lambda x}(x+\xi_{\mathbf{e}})u(x)\mathrm{d}x\right) = \mathbf{P}_0 \left(\xi_{\mathbf{e}}\leq 0\right) \mathcal{L}(xu(x))(\lambda) \nonumber\\
		   	&\quad - \mathbf{E}_0\left(1_{\left\{\xi_{\mathbf{e}}<0\right\}}\int_0^{-\xi_{\mathbf{e}}} e^{-\lambda x}xu(x)\mathrm{d}x\right) -  \mathbf{E}_0\left(1_{\left\{\xi_{\mathbf{e}}<0\right\}}\int_{-\xi_{\mathbf{e}}}^\infty e^{-\lambda x} |\xi_{\mathbf{e}}| u(x)\mathrm{d}x\right)  \nonumber\\
		   	&\geq \mathbf{P}_0 \left(\xi_{\mathbf{e}}\leq 0\right) \mathcal{L}(xu(x))(\lambda) - 2\mathbf{E}_0(|\xi_{\mathbf{e}}|^{1-\delta}) \mathcal{L}(x^\delta u(x))(\lambda).
		   \end{align}
		   By taking $\delta<\gamma^{-1}$, we get from Lemma \ref{l6} and  \eqref{eq11'} that
		   \begin{align}\label{eq12}
		   	&\limsup_{\lambda \downarrow 0} \lambda \left( \mathbf{P}_0 \left(\xi_{\mathbf{e}}\leq 0\right) \mathcal{L}(xu(x))(\lambda) -  \int_0^\infty e^{-\lambda x}x \mathbf{E}_0\left[1_{\{ \xi_{\mathbf{e}}<0\}}u(x-\xi_{\mathbf{e}})\right]\mathrm{d}x\right)\nonumber\\
		   	&\leq 2\mathbf{E}_0(|\xi_{\mathbf{e}}|^{1-\delta})  \limsup_{\lambda \downarrow 0} \lambda \mathcal{L}(x^\delta u(x))(\lambda) \leq 2A\mathbf{E}_0(|\xi_{\mathbf{e}}|^{1-\delta})  \limsup_{\lambda \downarrow 0} \lambda  \int_0^\infty e^{-\lambda y} y^{\delta-1/\gamma}\mathrm{d} y\nonumber\\
		   	& = 2A\mathbf{E}_0(|\xi_{\mathbf{e}}|^{1-\delta})  \limsup_{\lambda \downarrow 0} \lambda^{1/\gamma -\delta} \int_0^\infty e^{-y} y^{\delta-1/\gamma}\mathrm{d} y=0.
		   \end{align}
		   Combining \eqref{to-prove-upper}, \eqref{eq9} and \eqref{eq12}, we get that \eqref{toprove-Laplace-limit} holds when $\alpha=1$.
		
		The rest of the proof is now similar to the case
		$\alpha \in(0,1)$.
		We first consider the case $m=1$. In this case, $\Phi_0(x)=0$.
		Combining
		 \eqref{remainder1.2} and
			 Lemma \ref{Important-lemma2} (i)
		 with $f=u^{\gamma}$ and $f=u^{\gamma+1}$, we get that
		\begin{align}\label{upper-phi-r-2}
			\mathcal{L}\left[x\Phi_R\right](\lambda) &\leq (1+\varepsilon)C_2(\gamma)\left( \mathcal{L}\left[xu^{\gamma}(x)\right](\lambda)+\mathbf{E}_0\left[S_{\mathbf{e}}\right]\mathcal{L}\left[u^{\gamma}\right](\lambda)\right)\nonumber \\
			&\quad +\frac{1}{\delta^{\gamma+1}}\left(  \mathcal{L}\left[xu^{\gamma+1}(x)\right](\lambda)+ \mathbf{E}_0\left[S_{\mathbf{e}}\right]\mathcal{L}\left[u^{\gamma+1}\right](\lambda) \right) .
		\end{align}
		For any $\varepsilon_1>0$, since $\lim _{x \rightarrow+\infty} u(x)=0$,
		there exists
				$A_2>0$
		such that $u(x) \leq \varepsilon_1$ for $x \geq A_2$.
		Similar to \eqref{u^{gamma+1}}, we have
		\begin{equation}\label{xu^{gamma+1}}
			\mathcal{L}\left[xu^{\gamma+1}(x)\right](\lambda)\leq A_{2}^2+\varepsilon_1 \mathcal{L}\left[xu^{\gamma}(x)\right](\lambda).
		\end{equation}
		Combining \eqref{toprove-Laplace-limit}, \eqref{upper-phi-r-2} and \eqref{xu^{gamma+1}}, we get
		\begin{align}
	 \frac{c_+\Gamma(2-\alpha)}{\alpha}  \leq
		\left((1+\varepsilon)C_2(\gamma)+\varepsilon_1/\delta^{\gamma+1}\right)
		\liminf _{\lambda \downarrow 0}\lambda^{2-\alpha} \mathcal{L}\left[xu^{\gamma}(x)\right](\lambda).
		\end{align}
	Letting $\varepsilon_1\to 0$ first and then $\varepsilon\to 0$, we get that
		\begin{align}\label{lowerbound2}
		 \frac{c_+\Gamma(2-\alpha)}{\alpha} \leq
	C_2(\gamma)	\liminf _{\lambda \downarrow 0}\lambda^{2-\alpha} \mathcal{L}\left[xu^{\gamma}(x)\right](\lambda).
		\end{align}
	Next, combining \eqref{xu^{gamma+1}}, Lemma \ref{Important-lemma2} (ii) with $f=u^{\gamma}$ and Lemma \ref{Important-lemma2} (i) with $f=u^{\gamma+1}$, it holds that
		\begin{align*}
			& \mathcal{L}\left[x\Phi_R(x)\right](\lambda) \geq C_2(\gamma)(1-\varepsilon)\int_0^{+\infty} e^{-\lambda x} x\left(\mathbf{P}_0\left(\xi_{\mathbf{e}} \geq x\right)-\mathbf{P}_0\left(S_{\mathbf{e}} \geq x\right)\right) \mathrm{d} x\\
			& +C_2(\gamma)(1-\varepsilon)\left\{\mathbf{E}_0\left[e^{-\lambda \xi_{\mathbf{e}}^{+}}\right] \mathcal{L}\left[x u^{\gamma}(x)\right](\lambda)-\mathbf{E}_0\left[1_{\left\{\xi_{\mathbf{e}}<0\right\}} \int_0^{-\xi_{\mathbf{e}}} e^{-\lambda z} z u^{\gamma}(z)  \mathrm{d} z\right]+ \mathbf{E}_0\left[\xi_{\mathbf{e}} e^{-\lambda \xi_{\mathbf{e}}^{+}}\right] \mathcal{L}\left[u^{\gamma}\right](\lambda)\right\}\\
			& -\frac{1}{\delta^{\gamma+1}} \left\{A_{1}^2+\varepsilon_1 \mathcal{L}\left[xu^{\gamma}(x)\right](\lambda)+ \mathbf{E}_0\left[S_{\mathbf{e}}\right]\mathcal{L}\left[u^{\gamma+1}\right](\lambda)\right\}.
		\end{align*}
		Applying Lemma \ref{Important-lemma2} (iii) with $f=u^{\gamma}$, we obtain
		$$
		 \frac{c_+\Gamma(2-\alpha)}{\alpha} \geq
		\left(C_2(\gamma)(1-\varepsilon)-\varepsilon_1 /\delta^{\gamma+1}\right) \limsup _{\lambda \downarrow 0}\lambda^{2-\alpha}\mathcal{L}\left[xu^{\gamma}(x)\right](\lambda).
		$$
	Letting $\varepsilon_1\to 0$ first and then $\varepsilon\to 0$, we get that
		\begin{align}\label{upperbound2}
		\frac{c_+\Gamma(2-\alpha)}{\alpha}\geq C_2(\gamma) \limsup _{\lambda \downarrow 0}\lambda^{2-\alpha} \mathcal{L}\left[xu^{\gamma}(x)\right](\lambda).
		\end{align}
	 Combining \eqref{lowerbound2} and \eqref{upperbound2}, we conclude that
		$$
		\lim_{\lambda \downarrow 0}\frac{\mathcal{L}\left[x u^{\gamma}(x)\right](\lambda)}{\lambda^{\alpha-2}}=
		   \frac{c_+\Gamma(2-\alpha)}{C_2(\gamma)\alpha}.
		$$
		Finally, by the Tauberian theorem, we get
		$$
		\lim_{x \rightarrow+\infty}\frac{1}{x^{2-\alpha}}
		\int_0^x z u^{\gamma}(z) \mathrm{d} z= \frac{c_+}{\alpha(2-\alpha)C_2(\gamma)}.
		$$
		The desired result now follows from  Karamata's monotone density theorem \cite[Theorem 1.7.2]{bingham_goldie_teugels_1987}.

		Next we consider the  case  $m<1$. Combining $\Phi_R\geq 0$, \eqref{varepsilon_prime}
		and Lemma \ref{Important-lemma2} (i) with $f=u$,
		we get  that for any $\varepsilon^{\prime}>0$, there exists a constant $A_3=A_3(\varepsilon^{\prime})$, such that for all small $\lambda$,
		$$
		\mathcal{L}\left[x\Phi_0(x)\right](\lambda)  \leq \mathcal{L}\left[x(\Phi_0+\Phi_R)\right](\lambda)  \leq (1+\varepsilon^{\prime})(1-m)\left\{\mathcal{L}[xu(x)](\lambda)+ \mathbf{E}_0\left[S_{\mathbf{e}}\right]\mathcal{L}[u](\lambda)\right\}+A_3,
		$$
		which, by \eqref{toprove-Laplace-limit},  implies that
		\begin{equation}\label{lowerbound}
		(1-m)	\limsup _{\lambda \downarrow 0}\lambda^{2-\alpha} \mathcal{L}[xu(x)](\lambda)\leq 	 \frac{c_+\Gamma(2-\alpha)}{\alpha} \leq
		(1-m)(1+\varepsilon')	\liminf _{\lambda \downarrow 0}\lambda^{2-\alpha} \mathcal{L}[xu(x)](\lambda) .
		\end{equation}
			Letting $\varepsilon'\downarrow 0$, we get that
		$$
		\lim_{\lambda \downarrow 0}
			\frac{1}{\lambda^{\alpha-2}}\mathcal{L}\left[xu(x)\right](\lambda)=\frac{c_+\Gamma(2-\alpha)}{\alpha(1-m)}.
		$$
		Finally, by the Tauberian theorem,
		$$
		\lim_{x \rightarrow+\infty}
		\frac{1}{x^{2-\alpha}}\int_0^x z u(z) \mathrm{d} z= \frac{c_+\Gamma(2-\alpha)}{\alpha(1-m)}.
		$$
		The desired result now follows from  Karamata's monotone density theorem \cite[Theorem 1.7.2]{bingham_goldie_teugels_1987}.
		\end{proof}
		
	\section{Proof of Theorem \ref{t2}}
	
	Define $\widetilde{\xi}=-\xi$, $\widetilde{\tau}_y:= \inf\left\{t: \widetilde{\xi}_t\leq y\right\}$ and
	\[
	f(u) := \frac{G(u)}{u}= \frac{\sum_{k=0}^\infty p_k (1-u)^k -(1-u)}{u}, \quad u\in [0, 1].
	\]
	According to \cite[Lemma 2.3]{HJRS2025}, $u$ has the following representation:
	 \begin{align}\label{FK-formula}
	 	u(x)= \mathbf{E}_x\left(\exp\left\{-\int_0^{\widetilde{\tau}_y}f\left(u\left(\widetilde{\xi}_s\right)\right)\mathrm{d}s\right\}\right) u(y), \quad 0\leq y <x.
	 \end{align}
	 Similar to  \cite{HJRS2025}, we consider the function
	  \[
	  [0,\infty) \ni x\mapsto \frac{u\left(x+ y u(x)^{-\frac{\gamma-1}{\alpha}}\right)}{u(x)},
	  \]
	  which is bounded between $0$ and $1$.  Therefore, using a diagonalization argument, we can find a subsequence
	  $\{x_k\in [0,\infty)\}$ with $\lim_{k\to\infty} x_k =+\infty$ such that for all $y\geq 0, y\in \mathbb{Q}$, the following limits exist:
	  \begin{align}\label{step_7}
	  	\phi(y):= \lim_{k\to\infty} \frac{u\left(x_k+ y u(x_k)^{-\frac{\gamma-1}{\alpha}}\right)}{u(x_k)}.
	  \end{align}
	  Since $u(x)$ is decreasing, we see that $\phi(0)=1$ and $\phi(y)\in [0,1]$ for any $y\in  \mathbb{Q}\cap [0,\infty)$. Moreover,  $\phi$ is decreasing in $\mathbb{Q}\cap [0,\infty)$. Therefore, for any $y\geq 0$, we can define
	  \begin{align}\label{step_8}
	  	\phi(y):= \sup_{z\in \mathbb{Q}, z\geq y} \phi(z) = \lim_{z\in \mathbb{Q}, z\downarrow y} \phi(y).
	  \end{align}
	
	  \begin{lemma}\label{l:lemma4}
	  	  The limit \eqref{step_7} holds for all $y\geq 0$.
	  	  Also, it holds that for any $K>0$,
	  	  \begin{align}\label{step_5}
	  	  \lim_{k\to\infty} \sup_{y\in [0,K]} \left|  \frac{u\left(x_k+ y u(x_k)^{-\frac{\gamma-1}{\alpha}}\right)}{ \phi (y)u(x_k)} -1\right| =0.
	  	  \end{align}
	  	  Moreover,  $\phi$  satisfies the equation
	  	  \begin{align}\label{PDE}
	  	  	\phi(y) = \mathbf{E}_{y}\left(\exp\left\{-C_2(\gamma)\int_0^{\widetilde{\tau}_0}\left(\phi(\widetilde{\xi}_s)\right)^{\gamma-1}\mathrm{d}s\right\} \right),\quad y\geq 0,
	  	  \end{align}
	  	  where $C_2(\gamma)$ is defined in \eqref{step_2}.
	  \end{lemma}
	  \begin{proof}
	  	 Fix two arbitrary non-negative rational numbers $y_1<y_2$ and set $z_i(k)= y_i u(x_k)^{-\frac{\gamma-1}{\alpha}}$, $i=1, 2$. Combining the definition of $\phi$ and \eqref{FK-formula}, we get that
	  	 \begin{align}\label{step_1}
	  	 	\phi(y_1) &  \geq \phi(y_2) = \lim_{k\to\infty} \frac{u\left(x_k+ z_2(k)\right)}{u(x_k)} \nonumber\\
	  	 	& = \lim_{k\to\infty}  \mathbf{E}_{x_k+ z_2(k)}\left(\exp\left\{-\int_0^{\widetilde{\tau}_{x_k+z_1(k)}}f\left(u\left(\widetilde{\xi}_s\right)\right)\mathrm{d}s\right\}\right)  \frac{u\left(x_k +z_1(k)\right) }{u(x_k)}\nonumber\\
	  	 	& = \phi(y_1) \lim_{k\to\infty}  \mathbf{E}_{z_2(k)}\left(\exp\left\{-\int_0^{\widetilde{\tau}_{z_1(k)}}f\left(u\left(x_k +\widetilde{\xi}_s\right)\right)\mathrm{d}s\right\}\right).
	  	 \end{align}
Combining the scaling property of $\widetilde{\xi}$, \eqref{step_2} and \eqref{step_1}, we get that there exists $c>0$ such that
	  	 \begin{align}
	  	 	\phi(y_1) &  \geq \phi(y_2) \geq \phi(y_1) \lim_{k\to\infty}  \mathbf{E}_{z_2(k)}\left(\exp\left\{- c\int_0^{\widetilde{\tau}_{z_1(k)}}\left(u\left(x_k +\widetilde{\xi}_s\right)\right)^{\gamma-1}\mathrm{d}s\right\}\right)\nonumber\\
	  	 	& = \phi(y_1)\lim_{k\to\infty}  \mathbf{E}_{y_2 }\left(\exp\left\{- c\int_0^{u(x_k)^{-(\gamma-1)}\widetilde{\tau}_{y_1}}\left(u\left(x_k +\widetilde{\xi}_{s u(x_k)^{\gamma-1}} u(x_k)^{-\frac{\gamma-1}{\alpha}}\right)\right)^{\gamma-1}\mathrm{d}s\right\}\right)\nonumber\\
	  	 	& = \phi(y_1)\lim_{k\to\infty}  \mathbf{E}_{y_2 }\left(\exp\left\{- c\int_0^{\widetilde{\tau}_{y_1}}\left(\frac{u\left(x_k +\widetilde{\xi}_{s } u(x_k)^{-\frac{\gamma-1}{\alpha}}\right)}{u(x_k)}\right)^{\gamma-1}\mathrm{d}s\right\}\right).
	  	 \end{align}
	  	 Since $u\left(x_k +\widetilde{\xi}_{s } u(x_k)^{-\frac{\gamma-1}{\alpha}}\right)\leq u(x_k)$ for all $s\leq \widetilde{\tau}_{y_1}$, we get from the above inequality that
	  	 \begin{align}\label{step_3}
	  	 		\phi(y_1) &  \geq \phi(y_2) \geq \phi(y_1)  \mathbf{E}_{y_2 }\left(\exp\left\{- c\widetilde{\tau}_{y_1}\right\}\right) =  \phi(y_1)  \mathbf{E}_{1}\left(\exp\left\{- c(y_2-y_1)^\alpha\widetilde{\tau}_{0}\right\}\right),
	  	 \end{align}
	  	 where in the last equality we also used the scaling property of $\widetilde{\xi}$.
		 Therefore, for any $y> 0$ and  any positive rational number $y_1 \leq y<y_2$, we have
	  	 \begin{align}\label{step_4}
	  	 	\phi(y_2) & =\lim_{k\to\infty} \frac{u\left(x_k+ y_2 u(x_k)^{-\frac{\gamma-1}{\alpha}}\right)}{u(x_k)}\leq \liminf_{k\to\infty} \frac{u\left(x_k+ y u(x_k)^{-\frac{\gamma-1}{\alpha}}\right)}{u(x_k)}\nonumber\\
	  	 	& \leq \limsup_{k\to\infty} \frac{u\left(x_k+ y u(x_k)^{-\frac{\gamma-1}{\alpha}}\right)}{u(x_k)}\leq \lim_{k\to\infty} \frac{u\left(x_k+ y_1 u(x_k)^{-\frac{\gamma-1}{\alpha}}\right)}{u(x_k)}=\phi(y_1).
	  	 \end{align}
	  	 Combining \eqref{step_3} and \eqref{step_4}, we see that \eqref{step_7} holds for all $y\geq 0$.
	  	
	  	 Taking $y_1=0$ in \eqref{step_3}, we see that $\inf_{y\in [0,K]} \phi (y)>0$.
		 Therefore, using an argument similar to that leading to \cite[(3.9)]{HJRS2025}, we can get \eqref{step_5}.
	  	
	  	 Now we prove \eqref{PDE}. Since $u(x_k +\widetilde{\xi}_s)\leq u(x_k)$ for any $s\leq \widetilde{\tau}_{0}$, combining \eqref{step_2} and \eqref{FK-formula}, we see that for any $\varepsilon >0$,
	  		  	 \begin{align}
	  	 	\phi(y) &= \lim_{k\to\infty} \frac{u\left(x_k+ y u(x_k)^{-\frac{\gamma-1}{\alpha}}\right)}{u(x_k)} = \lim_{k\to\infty} \mathbf{E}_{y u(x_k)^{-\frac{\gamma-1}{\alpha}} }\left(\exp\left\{-\int_0^{\widetilde{\tau}_0}f\left(u\left(x_k +\widetilde{\xi}_s\right)\right)\mathrm{d}s\right\}\right)\nonumber\\
	  	 	& \geq \lim_{k\to\infty} \mathbf{E}_{y u(x_k)^{-\frac{\gamma-1}{\alpha}} }\left(\exp\left\{-C_2(\gamma)(1+\varepsilon)\int_0^{\widetilde{\tau}_0}\left(u\left(x_k + \widetilde{\xi}_s\right)\right)^{\gamma-1}\mathrm{d}s\right\}\right).
	  	 \end{align}
	  	Applying the  scaling property again,  we get from the above inequality that
	  	\begin{align}
	  		\phi(y) &\geq \lim_{k\to\infty} \mathbf{E}_{y}\left(\exp\left\{-C_2(\gamma)(1+\varepsilon)\int_0^{\widetilde{\tau}_0}\left(\frac{u\left(x_k +  \widetilde{\xi}_su(x_k)^{-\frac{\gamma-1}{\alpha}} \right)}{u(x_k)}\right)^{\gamma-1}\mathrm{d}s\right\}\right)\nonumber\\
	  		& = \mathbf{E}_{y}\left(\exp\left\{-C_2(\gamma)(1+\varepsilon)\int_0^{\widetilde{\tau}_0}\left(\phi(\widetilde{\xi}_s)\right)^{\gamma-1}\mathrm{d}s\right\} \right),
	  	\end{align}
	  	where in the last equality we used the dominated convergence theorem. Letting $\varepsilon\downarrow 0$, we get the lower bound. The proof of the upper bound is similar and we omit the details.
	  \end{proof}
	
As a consequence of Lemma \ref{l:lemma4}, we see that $\phi(y)\in (0, 1)$ for all $y>0$.
	
	  \begin{proposition}\label{prop1}
	  	  (i) $\phi$ is the unique solution of of \eqref{PDE}.
	  	
	  	  (ii) The equation
	  	  \begin{align}\label{Another-PDE}
	  	  		U(z) = \mathbf{E}_{z}\left(\exp\left\{-C_2(\gamma)\int_0^{\widetilde{\tau}_y}\left(U(\widetilde{\xi}_s)\right)^{\gamma-1}\mathrm{d}s\right\} \right)U(y),\quad z> y\geq 0
	  	  \end{align}
	  	  has a unique solution satisfying  the boundary conditions $\lim_{y\downarrow 0} U(y)=+\infty$ and $\lim_{y\to\infty} U(y)=0$.
	  \end{proposition}
	
The proof of Proposition \ref{prop1} is postponed to Section \ref{s:5}.

	  \begin{lemma}\label{test-lemma}
	  	 It holds that
	  	 \[
	  	 \limsup_{x\to\infty} x^{\frac{\alpha}{\gamma-1}} u(x) =  \limsup_{x\to\infty} x^{\frac{\alpha}{\gamma-1}} \mathbb{P}\left(M\geq x\right) <\infty.
	  	 \]
	  	 and that
	  	  \[
	  	 \liminf_{x\to\infty} x^{\frac{\alpha}{\gamma-1}} u(x) =  \liminf_{x\to\infty} x^{\frac{\alpha}{\gamma-1}} \mathbb{P}\left(M \geq x\right) >0.
	  	 \]
	  \end{lemma}
	  \begin{proof}
	  	Define $w(x)= x^{\frac{\alpha}{\gamma-1}} u(x)$ and set $A= \liminf_{x\to +\infty} w(x)$ and $B= \limsup_{x\to\infty} w(x)$.  It suffices to show that $0<A \leq B<\infty$.
	  	
	  	 Combining \eqref{step_5} and Proposition \ref{prop1},  we get that, for any $K>0$,
	  	  \begin{align}\label{step_6}
	  	  		0 & = \lim_{x\to+\infty} \sup_{y\in [0,K]} \left|  \frac{u\left(x+ y u(x)^{-\frac{\gamma-1}{\alpha}}\right)}{ \phi (y)u(x)} -1\right| \nonumber\\
	  	  		&=\lim_{x\to\infty} \sup_{y\in [0,K]} \left|  \frac{w\left(x \left(1+ y w(x)^{-\frac{\gamma-1}{\alpha}}\right)\right)}{ \phi (y)w(x) \left(1+ yw(x)^{-\frac{\gamma-1}{\alpha}}\right)^{\frac{\alpha}{\gamma-1}}} -1\right| .
	  	  \end{align}
	  	  First we show that $A<\infty$. If $A=\infty$, we define $b_k:= \sup\left\{x: w(x)< k\right\}$. Then $b_k\to\infty$ as $k\to\infty$.
		  Using the the definition of $b_k$ and the left-continuity of $w$, we get that $w(b_k)\leq k\leq \inf_{z>b_k} w(z)$.
		  Therefore, taking $x=b_k$ in \eqref{step_6}, we get that for any $y>0$,
	  	  \begin{align}
	  	  	1& = \lim_{k\to\infty} \frac{w\left(b_k \left(1+ y w(x)^{-\frac{\gamma-1}{\alpha}}\right)\right)}{ \phi (y)w(b_k) \left(1+ yw(b_k)^{-\frac{\gamma-1}{\alpha}}\right)^{\frac{\alpha}{\gamma-1}}} \geq  \lim_{k\to\infty} \frac{k}{ \phi (y) \left(w(b_k)^{\frac{\gamma-1}{\alpha}}+ y\right)^{\frac{\alpha}{\gamma-1}}}\nonumber\\
	  	  	& \geq  \lim_{k\to\infty} \frac{k}{ \phi (y) \left(k^{\frac{\gamma-1}{\alpha}}+ y\right)^{\frac{\alpha}{\gamma-1}}} = \frac{1}{\phi(y)},
	  	  \end{align}
	  	  which is a contraction to
		  Lemma \ref{l:lemma4}.
		  Hence $A<\infty$. Similarly we can show that $B>0$.
	  	
	  	  Now we show that $B<\infty$. Assume that $A<B=\infty$.
		  Note that for any $K>0$,
	  	  \begin{align}
	  	  	 \lim_{A_1 \to\infty} \phi(K) \left(1+ K A_1^{-\frac{\gamma-1}{\alpha}}\right)^{\frac{\alpha}{\gamma-1}} = \phi(K)<1.
	  	  \end{align}
	  	  Therefore, we may fix an $A_1>A$ and an $\varepsilon>0$ such that
	  	  \begin{align}\label{step_10}
	  	  (1+\varepsilon)	\phi(K) \left(1+ K A_1^{-\frac{\gamma-1}{\alpha}}\right)^{\frac{\alpha}{\gamma-1}} <1.
	  	  \end{align}
	  	   Fix another $B_1>A_1$. Define
	  	  \begin{align}
	  	  	a_1& := \inf\left\{x>0: w(x) <A_1\right\},\quad d_1:= \inf\left\{x>a_1: w(x) >B_1+1\right\},\nonumber\\
	  	  	a_k &:= \inf\left\{x>d_{k-1}: w(x)< A_1 \right\},\quad d_k:=\inf\left\{x>a_k: w(x)> B_1+k\right\},\nonumber\\
	  	  	a_k^*&:= \sup\left\{x\in [a_k, d_k]: w(x)< A_1\right\}.
	  	  \end{align}
	  	  By \eqref{step_6}, for any $\varepsilon, K>0$ satisfying \eqref{step_10}, there exists $N>0$ such that  when $x>N$.
	  	  \begin{align}\label{step_9}
	  	  	\sup_{y\in [0,K]} \left|  \frac{w\left(x \left(1+ y w(x)^{-\frac{\gamma-1}{\alpha}}\right)\right)}{ \phi (y)w(x) \left(1+ yw(x)^{-\frac{\gamma-1}{\alpha}}\right)^{\frac{\alpha}{\gamma-1}}} -1\right|  <\varepsilon.
	  	  \end{align}
	  	  By the left continuity of $w$, taking $x=a_k^*$ in \eqref{step_9}, we see that when $k$ is large enough, for all $y\in [0, K]$,
	  	  \begin{align}\label{step_11}
	  	  	w\left(a_k^* \left(1+ y w(a_k^*)^{-\frac{\gamma-1}{\alpha}}\right)\right) & \leq (1+\varepsilon) \phi (y)w(a_k^*) \left(1+ yw(a_k^*)^{-\frac{\gamma-1}{\alpha}}\right)^{\frac{\alpha}{\gamma-1}} \nonumber\\
	  	  	& = (1+\varepsilon) \phi (y) \left(w(a_k^*)^{\frac{\gamma-1}{\alpha}}+ y\right)^{\frac{\alpha}{\gamma-1}} \leq (1+\varepsilon) \phi(y) \left(A_1^{\frac{\gamma-1}{\alpha}}+ K\right)^{\frac{\alpha}{\gamma-1}} \nonumber\\
	  	  	&<B_1+k.
	  	  \end{align}
	  	  Therefore, we see that when $k$ is large enough,
	  	  \begin{align}\label{step_12}
	  	  	  \left\{  a_k^* \left(1+ y w(a_k^*)^{-\frac{\gamma-1}{\alpha}}\right): y\in [0, K]\right\} \subset [a_k^*, d_k] \quad \Longrightarrow \quad w\left(a_k^* \left(1+ K w(a_k^*)^{-\frac{\gamma-1}{\alpha}}\right)\right)  \geq A_1.
	  	  \end{align}
	  	  However, combining \eqref{step_10} and \eqref{step_11}, we have
	  	  \begin{align}
	  	  	w\left(a_k^* \left(1+ K w(a_k^*)^{-\frac{\gamma-1}{\alpha}}\right)\right)  \leq (1+\varepsilon) \phi(K)  \left(A_1^{\frac{\gamma-1}{\alpha}}+ K\right)^{\frac{\alpha}{\gamma-1}}  < A_1,
	  	  \end{align}
	  	  which contradicts \eqref{step_12}. Therefore,  $B<\infty$.
	  	
	  	  Now we prove $A>0$.  If $A=0<B$,
		  then
		  combining \eqref{step_3} and \eqref{Laplace-of-stopping-time}, we see that for any $y>0$,
	  	  \begin{align}
	  	  	 \phi(y)\geq \mathbf{E}_0\left( \exp\left\{-c y^\alpha \widetilde{\tau}_{-1} \right\}\right)= e^{-c_1 y}
	  	  \end{align}
	  	  for some constant $c_1>0$, where $c$ is the constant in  \eqref{step_3}. Note that for any $B_2>0$,
	  	  \begin{align}
	  	  	  \phi(y)^{\frac{\gamma-1}{\alpha}} \left(1+y B_2^{-\frac{\gamma-1}{\alpha}}\right) & \geq e^{-c_1 \frac{\gamma-1}{\alpha} y} \left(1+y B_2^{-\frac{\gamma-1}{\alpha}}\right) \geq \left(1-c_1 \frac{\gamma-1}{\alpha} y \right)\left(1+y B_2^{-\frac{\gamma-1}{\alpha}}\right) \nonumber\\
	  	  	  & = 1+ \left( B_2^{-\frac{\gamma-1}{\alpha}} -
	  	  	  c_1
	  	  	  \frac{\gamma-1}{\alpha} \right) y -c_1 \frac{\gamma-1}{\alpha}B_2^{-\frac{\gamma-1}{\alpha}}y^2.
	  	  \end{align}
	  	  Thus we  can take $B_2$ and $K$ sufficiently small so that for all $y\in (0, K]$,
	  	  \[
	  	  \phi(y)^{\frac{\gamma-1}{\alpha}} \left(1+y B_2^{-\frac{\gamma-1}{\alpha}}\right)>1.
	  	  \]
		  Let $B_2$ and $K$ be chosen as above and fix $y\in (0, K]$. Then there exists $\varepsilon>0$ such that
	  	    \begin{align}\label{step_10'}
	  	  (1-\varepsilon) \phi (y)\left(1+ yB_2^{-\frac{\gamma-1}{\alpha}}\right)^{\frac{\alpha}{\gamma-1}} >1.
	  	  \end{align}
		  Take $N$ large enough so that $N^{-1}< B_2$ and define
	  	   \begin{align}
	  	  	 h_1& := \inf\left\{x>0: w(x) >B_2 \right\},\quad j_1:= \inf\left\{x>h_1: w(x) < \frac{1}{N+1} \right\},\nonumber\\
	  	  	 h_k &:= \inf\left\{x>j_{k-1}: w(x)> B_2 \right\},\quad j_k:=\inf\left\{x>h_k: w(x)< \frac{1}{N+k}\right\},\nonumber\\
	  	  	 h_k^*&:= \sup\left\{x\in [h_k, j_k]: w(x)>B_2\right\}.
	  	  \end{align}
	  	  Combining \eqref{step_9} and the left continuity of $w$,  we see that $w(h_k^*)\geq B_2$ and that,  when $k$ is large enough, for any $y\in [0, K]$,
	  	   \begin{align}\label{step_13}
	  	  	w\left(h_k^* \left(1+ y w(h_k^*)^{-\frac{\gamma-1}{\alpha}}\right)\right) & \geq (1-\varepsilon) \phi (y)w(h_k^*) \left(1+ yw(h_k^*)^{-\frac{\gamma-1}{\alpha}}\right)^{\frac{\alpha}{\gamma-1}} \nonumber\\
	  	  	& = (1-\varepsilon) \phi (y) \left(w(h_k^*)^{\frac{\gamma-1}{\alpha}}+ y\right)^{\frac{\alpha}{\gamma-1}} \geq (1-\varepsilon)  \phi(K)B_2>\frac{1}{k+N},
	  	  \end{align}
	  	  which implies that
	  	    \begin{align}\label{step_12'}
	  	  	\left\{  h_k^* \left(1+ y w(h_k^*)^{-\frac{\gamma-1}{\alpha}}\right): y\in [0, K]\right\} \subset [h_k^*, j_k] \quad \Longrightarrow \quad w\left(h_k^* \left(1+ K w(h_k^*)^{-\frac{\gamma-1}{\alpha}}\right)\right)  \leq  B_2.
	  	  \end{align}
	  	  However, combining \eqref{step_10'} and \eqref{step_13}, we get
	  	   \begin{align}\label{step_14}
	  	  	w\left(h_k^* \left(1+ K w(h_k^*)^{-\frac{\gamma-1}{\alpha}}\right)\right) & \geq (1-\varepsilon) \phi (K)B_2\left(1+ KB_2^{-\frac{\gamma-1}{\alpha}}\right)^{\frac{\alpha}{\gamma-1}} >B_2,
	  	  \end{align}
	  	  which contradicts \eqref{step_12'}. Therefore, $A>0$ and the proof is copmplete.
	  \end{proof}

	  Now we are ready to prove Theorem \ref{t2}.
	
	  \begin{proof}[Proof of Theorem \ref{t2}]
	  	  Define $U^{(x)}(y):= x^{\frac{\alpha}{\gamma-1}}u(xy)$, then it follows
		  from Lemma \ref{test-lemma} that for some constant $\gamma_1, \gamma_2$ and $A$, it holds that
	  	  \begin{align}\label{Rough-bound}
	  	  	 \frac{\gamma_1}{y^{\frac{\alpha}{\gamma-1}}}\leq U^{(x)}(y)\leq \frac{\gamma_2}{y^{\frac{\alpha}{\gamma-1}}},\quad xy\geq A.
	  	  \end{align}
		  It follows from \eqref{step_2} that for any $u_0\in (0, 1)$,
		  $f(u)/u^{\gamma-1}\leq c$ for all $u\in [0, u_0]$ for some positive constant $c$.
	  	  Fixing a $y_0>0$. Then by \eqref{FK-formula}, when $x> A/y_0$,  we see that for any $z>y>y_0$, under $\mathbf{P}_{xz}$, we have $u(\widetilde{\xi}_s)\leq \frac{\gamma_2}{\widetilde{\xi}_s^{\frac{\alpha}{\gamma-1}}}\leq \frac{\gamma_2}{(xy_0)^{\frac{\alpha}{\gamma-1}}}$ on the set $\{s< \widetilde{\tau}_{xy}\}$. Therefore, by \eqref{Laplace-of-stopping-time},
	  	  for any $y_0<y<z$,
	  	  \begin{align}\label{step_16}
	  	  	 U^{(x)}(z) & =  \mathbf{E}_{xz}\left(\exp\left\{-\int_0^{\widetilde{\tau}_{xy}}f\left(u\left(\widetilde{\xi}_s\right)\right)\mathrm{d}s\right\}\right) U^{(x)}(y)\nonumber\\
	  	  	 & \geq \mathbf{E}_{xz}\left(\exp\left\{- c\int_0^{\widetilde{\tau}_{xy}}\left(u\left(\widetilde{\xi}_s\right)\right)^{\gamma-1}\mathrm{d}s\right\}\right) U^{(x)}(y)\nonumber\\
	  	  	 &\geq \mathbf{E}_{xz}\left(\exp\left\{- \frac{c\gamma_2^{\gamma-1}}{(xy_0)^{\alpha}}\widetilde{\tau}_{xy} \right\}\right) U^{(x)}(y)= \mathbf{E}_{1}\left(\exp\left\{- \frac{c\gamma_2^{\gamma-1}}{(y_0)^{\alpha}}(z-y)^\alpha\widetilde{\tau}_{0} \right\}\right) U^{(x)}(y)\nonumber\\
	  	  	 & = e^{-c_1 (z-y)} U^{(x)}(y),
	  	  \end{align}
for some positive constant	 $c_1$. 	
	  	  Therefore, combining \eqref{Rough-bound} and \eqref{step_16}, we see that when $x>A/y_0$, for any $y_0<y<z$,
	  	  \begin{align}\label{step_15}
	  	  	 \left| U^{(x)}(z) -U^{(x)}(y) \right| = U^{(x)}(y) -U^{(x)}(z)\leq U^{(x)}(y)\left( 1- e^{-c_1(z-y)}\right) \leq \frac{\gamma_2}{y_0^{\frac{\alpha}{\gamma-1}}} c_1 |z-y|.
	  	  \end{align}
	  	  Therefore,  combining \eqref{Rough-bound} and a diagonalization argument,  we can find sequence from $\{t_k: k\geq1\}\subset (0,\infty)$ with $t_k\to\infty$ such that
	  	  \begin{align}\label{Limit2}
	  	  	 U(y)= \lim_{k\to\infty} U^{(t_k)}(y),\quad \mbox{for all}\quad y\in \mathbb{Q}\cap (0,\infty).
	  	  \end{align}
		  Moreover, using a standard argument (for example, see \cite[Lemma 3.1]{HRS2024}) and with the help of \eqref{step_15}, one can show that \eqref{Limit2} holds for all $y>0$.
		taking  $x=t_k$ in \eqref{Rough-bound} and letting $k\to\infty$, we  see that
		$U(y)$ is comparable to $y^{-\frac{\alpha}{\gamma-1}}$, which implies that $\lim_{y\downarrow 0} U(y)=\infty$ and that $\lim_{y\to\infty} U(y)=0$.

	  	  Let $z>y\geq y_0$. Since $u(\widetilde{\xi}) \leq u(xy)$ for all $s\leq \widetilde{\tau}_{xy}$ under $\mathbf{P}_{xz}$, we get from  \eqref{step_2} that for any $\varepsilon, y_0>0$, when $x$ is large enough, we have,
	  	  \begin{align}
	  	  		 U^{(x)}(z) & \geq  \mathbf{E}_{xz}\left(\exp\left\{-C_2(\gamma)(1+\varepsilon)\int_0^{\widetilde{\tau}_{xy}}\left(u\left(\widetilde{\xi}_s\right)\right)^{\gamma-1}\mathrm{d}s\right\}\right) U^{(x)}(y)\nonumber\\
	  	  	&  = \mathbf{E}_{z}\left(\exp\left\{-C_2(\gamma)(1+\varepsilon)\int_0^{x^\alpha\widetilde{\tau}_{y}}\left(u\left(x\widetilde{\xi}_{s/x^\alpha}\right)\right)^{\gamma-1}\mathrm{d}s\right\}\right) U^{(x)}(y)\nonumber\\
	  	  	& = \mathbf{E}_{z}\left(\exp\left\{-C_2(\gamma)(1+\varepsilon)\int_0^{\widetilde{\tau}_{y}} x^\alpha\left(u\left(x\widetilde{\xi}_{s}\right)\right)^{\gamma-1}\mathrm{d}s\right\}\right) U^{(x)}(y) \nonumber\\
	  	  	& = \mathbf{E}_{z}\left(\exp\left\{-C_2(\gamma)(1+\varepsilon)\int_0^{\widetilde{\tau}_{y}} \left(U^{(x)}\left(\widetilde{\xi}_{s}\right)\right)^{\gamma-1}\mathrm{d}s\right\}\right) U^{(x)}(y) .
	  	  \end{align}
	  	  Taking $x=t_k$ in the above inequality and then letting $k\to\infty$, we get that
	  	  \begin{align}
	  	  	 U(z) &\geq \mathbf{E}_{z}\left(\exp\left\{-C_2(\gamma)(1+\varepsilon)\int_0^{\widetilde{\tau}_{y}} \left(U\left(\widetilde{\xi}_{s}\right)\right)^{\gamma-1}\mathrm{d}s\right\}\right) U(y) \nonumber\\
	  	  	 & \stackrel{\varepsilon\to 0}{\longrightarrow}\mathbf{E}_{z}\left(\exp\left\{-C_2(\gamma)\int_0^{\widetilde{\tau}_{y}} \left(U\left(\widetilde{\xi}_{s}\right)\right)^{\gamma-1}\mathrm{d}s\right\}\right) U(y).
	  	  \end{align}
	  	  Similarly, we also have for any $y_0\leq y<z$,
	  	  \begin{align}
	  	  	  U(z) &\leq \mathbf{E}_{z}\left(\exp\left\{-C_2(\gamma)\int_0^{\widetilde{\tau}_{y}} \left(U\left(\widetilde{\xi}_{s}\right)\right)^{\gamma-1}\mathrm{d}s\right\}\right) U(y).
	  	  \end{align}
	  	  Therefore, $U$ is the solution of the equation  in Proposition \ref{prop1}(ii). Now the constant $C_3(\alpha, \beta,\gamma)= U(1)$ follows immediately, we are done.
	  \end{proof}
	
	  \section{Proof of Theorem \ref{t3}}\label{s:5}

	  Repeating the proof of \cite[Lemma 2.3]{HJRS2025}, we can get that, in the case $m<1$,
	  $u$ has the following representation:
	  	 \begin{align}\label{FK-formula-sub}
	  		u(x)= \mathbf{E}_x\left(\exp\left\{- (1-m)\widetilde{\tau}_y-\int_0^{\widetilde{\tau}_y}f_{sub}\left(u\left(\widetilde{\xi}_s\right)\right)\mathrm{d}s\right\}\right) u(y),\quad x>y\geq 0
	  \end{align}
	  where
	  	\[
	  f_{sub}(u) = \frac{\sum_{k=0}^\infty p_k (1-u)^k -(1-m)(1-u)}{u}
	  \]
	  is a continuous function of $u\in (0,1]$ with $\lim_{u\to 0+}f_{sub}(u)=0$.  Moreover, by \cite[Lemma 2.7]{HRSZ2025}, $f_{sub}$ is increasing in $u$, and
that if $\sum_{k=0}^\infty k(\log k)p_k<\infty$,  then  for any $c>0$, $\int_0^\infty f_{sub} \left(e^{-ct} \right)\mathrm{d} t<\infty$, which is equivalent to
	  \begin{align}\label{Inte-of-f-sub}
		 \sum_{n=1}^\infty f_{sub} \left(e^{-cn} \right)<\infty.
	  \end{align}
	
	  \begin{proof}[Proof of Theorem \ref{t3}]
	  	  For simplicity, define $a_0:= \left((1-m)/C_1(\alpha)\right)^{1/\alpha}.$ According to \eqref{Laplace-of-stopping-time}, we have  for any $x>y\geq 0$,
	  	  \[
	  	  u(x)\leq  \mathbf{E}_x\left(\exp\left\{- (1-m)\widetilde{\tau}_y\right\}\right) u(y) = e^{-a_0(x-y)}u(y).
	  	  \]
	  	  Therefore, we see that $e^{a_0x} u(x)$ is decreasing in $x$ and that $\lim_{x\to\infty} e^{a_0x} u(x)\in [0,1].$ Thus, it remains to show that the limit is positive.
	  	
	  	  Taking $x=n+1, y=n$ in \eqref{FK-formula-sub}, we see that
	  	  \begin{align}
	  	  	u(n+1) = \mathbf{E}_{n+1}\left(\exp\left\{- (1-m)\widetilde{\tau}_n-\int_0^{\widetilde{\tau}_n}f_{sub}\left(u\left(\widetilde{\xi}_s\right)\right)\mathrm{d}s\right\}\right) u(n).
	  	  \end{align}
	  	  Since under $\mathbf{P}_{n+1}$, on the set $s<\widetilde{\tau}_n$, we have $\widetilde{\xi}_s \geq n$. Therefore, by the monotonicities of $f_{sub}$ and $u$, we conclude from the above identity that
	  	    \begin{align}\label{lower-of-u-n}
	  	  	u(n+1) & \geq  \mathbf{E}_{n+1}\left(\exp\left\{- (1-m)\widetilde{\tau}_n-\int_0^{\widetilde{\tau}_n}f_{sub}\left(u\left(n\right)\right)\mathrm{d}s\right\}\right) u(n)\nonumber\\
	  	  	& \geq \mathbf{E}_{n+1}\left(\exp\left\{- \left(1-m+f_{sub}\left(e^{-a_0n}\right)\right)\widetilde{\tau}_n\right\}\right) u(n)\nonumber\\
	  	  	& = \exp\left\{- H\left( 1-m+f_{sub}\left(e^{-a_0n}\right)\right)\right\} u(n),
	  	  \end{align}
	  	  where $H(a):= (a/C_1(\alpha))^{1/\alpha}$.
	  	  Noticing that for $\alpha\in (1,2)$, by Taylor's expansion,  we have
	  	  \[
	  	  H(1-m+v)= a_0+ H'(1-m) v+ O(v^2), \quad v\to0.
	  	  \]
	  	  Therefore, there exists $C>0$ such that for all $v\in (0,1)$, $ H(1-m+v)\leq  a_0+C v$.
Combining this inequality with \eqref{Inte-of-f-sub} and \eqref{lower-of-u-n},
we conclude that
	  	  \begin{align}
	  	  	 e^{a_0(n+1)}	u(n+1) & \geq   e^{a_0n} u(n) \exp\left\{- Cf_{sub}\left(e^{-a_0n}\right)\right\} \nonumber\\
	  	  	 &\geq \cdots \geq u(0) \exp\left\{-C\sum_{k=0}^n f_{sub}
                 \left(e^{-a_0k}\right) \right\}\nonumber\\
	  	  	 &\geq \exp\left\{-C\sum_{k=0}^\infty f_{sub}\left(e^{-a_0n}\right) \right\}>0,
	  	  \end{align}
	  	  which implies the desired result.
	  \end{proof}

	  \section{Proof of Proposition \ref{prop1}}\label{s:6}
	
	  The proof of Proposition \ref{prop1} relies heavily on another important Markov process: super $\alpha$-stable process. We will briefly introduce this process and some known results.
	
	  Let $\mathcal{M}_F(\mathbb R)$ be the families of finite Borel measures on $\mathbb R$. We will use $\mathbf{0}$ to denote  the null measure on $\mathbb R$. Let $B_b(\mathbb R)$ and $B_b^+(\mathbb R)$ be the spaces of bounded Borel functions and non-negative bounded Borel functions on $\mathbb R$ respectively.
	  For any $f\in B_b(\mathbb R)$ and $\mu\in \mathcal{M}_F(\mathbb R)$,
	  we use $\langle f, \mu\rangle$ to denote the integral of $f$ with respect to $\mu$. For any $\alpha\in (1, 2]$, the function
	  \begin{align}\label{Stable-Branching-mechanism}
	  	\varphi(\lambda) := 	C_2(\gamma) \lambda^\gamma
	  \end{align}
	    is a branching mechanism.
	
	   For any $\mu \in \mathcal{M}_F(\mathbb R)$,  we  use
	  $X=\{(X_t)_{t\geq 0}; \mathbb P_\mu\}$
	  to denote a super $\alpha$-stable process with spatial motion $\widetilde{\xi}$ and branching mechanism $\varphi$, that is, an $\mathcal{M}_F(\mathbb R)$-valued Markov process such that for any $f\in B_b^+(\mathbb R)$,
	  $$
	  -\log  \mathbb{E}_{\mu} \left(\exp\left\{ -\langle f, X_{t} \rangle\right\}\right)=
\langle v_f(t, \cdot), \mu\rangle,
	  $$
 where $(t, x)\mapsto v_f(t, x)$
is the unique locally bounded non-negative solution to
	  \begin{align}\label{Evolution-cumulant-semigroup}
	  	v_f(t, x)&=\mathbf{E}_x\left(f(\widetilde{\xi}_t) \right)- \mathbf{E}_y\Big(\int_0^t \varphi\left( v_f(t-s,\widetilde{\xi}_s)\right)\mathrm{d}s\Big).
	  \end{align}

  According to Dynkin \cite{E.B1.}, for any open set $Q$ of $\mathbb R$, there corresponds a  random measure $X_Q$ such that,
		  $\mu\in \mathcal{M}_F(\mathbb R)$ with $\textup{supp}\  \mu \subset Q$, and any  $f\in B_b^+(\mathbb R)$,
	  $$
	  \mathbb{E}_\mu \left(\exp\left\{-\langle f, X_Q \rangle\right\}\right) = \exp\big\{-\langle v_f^{Q} , \mu \rangle \big\},
	  $$
	  where
	  $v^{Q}_f(x)$
	  is the unique positive solution of the equation
	  \begin{align}\label{Equation-exit-measure}
	  	v_f^{Q} (x) & = \mathbf{E}_{x} \left(f(\widetilde{\xi}_\tau)\right)- \mathbf{E}_{x} \int_0^{\tau_Q} \varphi\big(v_f^{Q}(\widetilde{\xi}_r)\big)\mathrm{d} r,
	  \end{align}
	  with $\tau_Q:= \inf\left\{ r: \widetilde\xi_r \notin Q\right\}$.
	
	  \begin{proof}[Proof of Proposition \ref{prop1}]
Taking $Q= (0, \infty)$
and $f=1$ in \eqref{Equation-exit-measure},
we see that $\phi(x)$ is  the unique bounded solution of the following equation:
	  	 \begin{align}
	  	 		\phi(x)& = 1- \mathbf{E}_{x} \int_0^{\widetilde{\tau}_0}\varphi\big(\phi( \widetilde{\xi}_r)\big)\mathrm{d} r.
	  	 \end{align}
	  	 Since the above equation is equivalent to \eqref{PDE}, we complete the proof of (i).
	  	
	  	 Now we turn to the proof of (ii). Similarly, let $U$ be an solution to \eqref{Another-PDE} with boundary condition $U(0+)=\infty$ and $U(\infty)=0$. Noticing that for each fixed $y$,  \eqref{Another-PDE} is equivalent to
	  	  \begin{align}
	  	 	U(z)& = U(y)- \mathbf{E}_{z} \int_0^{\widetilde{\tau}_y}\varphi\big(U( \widetilde{\xi}_r)\big)\mathrm{d} r, \quad z>y>0.
	  	 \end{align}
	  	 Therefore, since $\widetilde{\xi}$ is spectrally positive,  we see that for $Q= (y,\infty)$,
is supported on $\{y\}$.
Therefore, from \eqref{Equation-exit-measure}, we conclude that
	  	 \begin{align}
	  	 	 U(z) = -\log \mathbb{E}_{\delta_z}\left(\exp\left\{-U(y) X_{(y,\infty)}(\left\{y\right\})\right\}\right) = -\log \mathbb{E}_{\delta_{z-y}}\left(\exp\left\{-U(y) X_{(0,\infty)}(\left\{0\right\})\right\}\right) ,
	  	 \end{align}
	  	 where in the last equality we used the spatial homogeneous property of
super $\alpha$-stable process. Therefore, replacing $z$ by $z+y$ first and then letting $y\to 0+$, we conclude that
	  	 \begin{align}
	  	 	 U(z)= \lim_{y\to0} U(z+y)= -\lim_{y\to0}\log \mathbb{E}_{\delta_{z}}\left(\exp\left\{-U(y) X_{(0,\infty)}(\left\{0\right\})\right\}\right)  = -\log \mathbb{P}_{\delta_z}\left(X_{(0,\infty)}(\{0\})=0\right),
	  	 \end{align}
	  	 which implies the desired result.
	  \end{proof}
	
	\noindent
	{\bf Acknowledgements:}
	We thank Yichao Huang for helpful discussions.

\vskip 0.2truein
\vskip 0.2truein

\noindent{\bf Haojie Hou:}   School of Mathematics and Statistics, Beijing Institute of Technology, Beijing 100081, P. R. China.
 Email: {\texttt
houhaojie@bit.edu.cn}

\smallskip

\noindent{\bf Yiyang Jiang:}  School of Mathematical Sciences, Peking
University,   Beijing, 100871, P.R. China. Email: {\texttt
	jyy.0916@stu.pku.edu.cn}

\smallskip

\noindent{\bf Yan-Xia Ren:} LMAM School of Mathematical Sciences \& Center for
Statistical Science, Peking
University,  Beijing, 100871, P.R. China. Email: {\texttt
yxren@math.pku.edu.cn}

\smallskip
\noindent {\bf Renming Song:} Department of Mathematics,
University of Illinois Urbana-Champaign,
Urbana, IL 61801, U.S.A.
Email: {\texttt rsong@illinois.edu}

	\end{document}